\documentclass [a4paper, 12pt]{amsart}

\usepackage{times}

\usepackage{amsmath,amsfonts,amsthm,amssymb,amscd}
\usepackage{graphicx}
\usepackage{epstopdf}
\usepackage{fancyhdr}
\usepackage[all]{xy}

\newtheorem{cor}{Corollary}[section]
\newtheorem{te}[cor]{Theorem}
\newtheorem{p}[cor]{Proposition}
\newtheorem{q}[cor]{Question}
\newtheorem{claim}[cor]{Claim}
\newtheorem{lemma}[cor]{Lemma}
\newtheorem*{Kap}{Kaplansky's Direct Finiteness Conjecture}
\theoremstyle{definition}
\newtheorem{de}[cor]{Definition}
\theoremstyle{remark}
\newtheorem{ob}[cor]{Observation}
\newtheorem{ex}[cor]{Example}

\newtheorem{nt}[cor]{Notation}

\newcommand{\cz}{\mathbb{C}}
\newcommand{\nz}{\mathbb{N}}
\newcommand{\zz}{\mathbb{Z}}
\newcommand{\qz}{\mathbb{Q}}
\newcommand{\rz}{\mathbb{R}}

\newcommand{\nr}{\mathcal{N}}
\newcommand{\unit}{\mathcal{U}}

\newcommand{\vp}{\varepsilon}

   {\begin{verse}%
     \small }%
   {\end{verse}}

\begin{document}

\title{Linear sofic groups and algebras}

\author{Goulnara Arzhantseva}\thanks{The first author was supported in part by the ERC grant ANALYTIC no. 259527, and by the Swiss NSF, under Sinergia grant CRSI22-130435. The second author was supported by
the Swiss NSF, under Sinergia grant CRSI22-130435.}
\address[G. Arzhantseva]{University of Vienna,
Faculty of Mathematics, Nordbergstra${\ss}$e 15, 1090 Wien, Austria}
\email{goulnara.arzhantseva@univie.ac.at}
\author{Liviu P\u aunescu}
\address[L. P\u aunescu]{University of Vienna,
Faculty of Mathematics, Nordbergstra${\ss}$e 15, 1090 Wien, Austria and Institute of Mathematics
of the Romanian Academy (on leave), 21 Calea Grivitei Street, 010702 Bucharest, Romania}
\email{liviu.paunescu@univie.ac.at}

\begin{abstract}
We introduce and systematically study linear sofic groups and linear sofic algebras.
This generalizes amenable and LEF groups and algebras. We prove that a group is linear sofic if and only if its
group algebra is linear sofic. We show that linear soficity
for groups is a priori weaker than soficity but stronger than weak
soficity. We also provide an alternative
proof of a result of Elek and Szabo which states that sofic groups
satisfy Kaplansky's direct finiteness conjecture. 
\end{abstract}

\maketitle

\section{Introduction}

Metric approximation properties for groups have received
considerable attention in the last years, mainly due to the
notions of \emph{hyperlinear} and \emph{sofic} groups. Hyperlinear
groups appeared in the context of Alain Connes' embedding
conjecture (1976) in operator algebra and were introduced by
Florin R\u adulescu \cite{Ra}. Sofic groups were introduced by
Misha Gromov \cite{Gr} in his study of symbolic algebraic
varieties in relation to the~Gottschalk surjunctivity conjecture
(1973) in topological dynamics. They were called sofic by Weiss
\cite{W}. Over the last years, various strong results have been
obtained for sofic groups in seemingly unrelated areas of
mathematics.  For instance, they have been at the heart of
developments on profinite topology of free groups, unimodular
random networks, diophantine approximations, linear cellular
automata, $L^2$-torsion, profinite equivalence relations, measure
conjugacy invariants, and continuous (in contrast to traditional
binary) logic.

These group properties can be stated in elementary algebraic terms, as
approximation properties, or in the language of ultraproducts, as
the existence of an embedding in a certain metric ultraproduct. We mainly
use the later technique due to simplicity in writing. For a
careful introduction to the subject, including ultraproducts
terminology, see \cite{Pe,Pe-Kw}.

Throughout the article, let $\omega$  be a non-principal (or free)
ultrafilter on $\nz$. In general, $(n_k)_k$ or $(m_k)_k$ denote
sequences of natural numbers tending to infinity. We denote
by $S_n$ the symmetric group of degree $n$, that is the group of
permutations on a set of $n$ elements. This group is endowed with
the normalized Hamming distance:
\[d_{Hamm}(p,q)=\frac1n\left\vert\{i\colon p(i)\neq q(i)\}\right\vert.\]

\begin{de}
A group $G$ is \emph{sofic} if there exist a sequence of natural
numbers $(n_k)_k$ and an injective group morphism from $G$ into
the metric ultraproduct $\Pi_{k\to\omega}(S_{n_k},d_{Hamm})$.

Such a morphism is called a \emph{sofic representation} of $G$.
\end{de}

The goal of our paper is to introduce soficity for algebras. We
shall approximate our algebras by matrix algebras endowed with a
distance provided by the rank. Two matrices are close in this
distance if they are equal, as linear transformations, on a large
subspace. This is in essence similar to the Hamming distance.
Therefore, we call the corresponding algebras (and groups,
respectively) \emph{linear sofic}. We refer the reader to Section~\ref{sec:def} for precise definitions.

Our main results about linear soficity are the following.

\begin{te}\label{mainresult}
A group $G$ is linear sofic if and only if its group algebra $\cz
G$ is linear sofic.
\end{te}

This has to be regarded in the light of recent developments in asymptotic geometry of
algebras and, more specifically, of group algebras, see~\cite{Gr08,Tullio} and references therein. In particular,
the known fact that a group is amenable if and only if its group algebra is amenable \cite{Samet,El,Barth}  is an evident  predecessor
of the above result.

Our proof of this theorem also provides an alternative proof
of Kaplansky's direct finiteness conjecture for sofic groups, a
result due to Elek and Szabo \cite{El-Sz1}.

In \cite{Gl-Ri} Glebsky and Rivera defined the notion of
\emph{weakly sofic}  group by replacing $(S_n,d_{Hamm})$ in the
definition of sofic groups by arbitrary finite groups equipped
with a bi-invariant metric. At present, quite a few is known about weakly sofic groups.

\begin{te}
Sofic groups are linear sofic, while linear sofic groups are
weakly sofic.
\end{te}

Our viewpoint on the approximation of algebras and groups  has given rise to a number of challenging difficulties.
For instance, the equivalence between the metric ultraproduct interpretation and the algebraic approach in the definition of approximation,
as well as the fundamental amplification trick, are easy in the sofic  and hyperlinear cases.
In the rank metric case, both properties are highly non-trivial.
We successfully resolve these issues by introducing the \emph{rank amplification} and by analyzing the tensor product of Jordan blocks, see Sections~\ref{sec:def} and~\ref{sec:rampl}.

Our approach leads to interesting phenomena (nonexistent in the classical sofic case)
when approximations by complex matrix algebras are replaced by those over a different field (or a sequence of fields), see Sections~\ref{sec:q} and~\ref{sec:lin}.
For instance, using the fundamental result of real semialgebraic geometry, so-called \emph{Positivstellensatz}\footnote{This is a semialgebraic analogue of  famous Hilbert's Nullstellensatz.},
we establish the equivalence between linear sofic representations over the field of complex numbers and those over
the rationals. 

Although Kaplansky's direct finiteness conjecture remains open for linear sofic groups (see Questions \ref{linsofkap} and \ref{q:oneF}),
we show that this new class of groups share with sofic groups several positive results. 
In particular, the class of linear sofic groups  is preserved under many group-theoretical operations, see Section~\ref{sec:perm}. Moreover,
under the failure of the Continuum Hypothesis, there exist $2^{\aleph_c}$ of \emph{universal linear sofic} groups,
up to metric isomorphism, see Section~\ref{sec:universal}.  

Our intention to study soficity of algebras is  motivated by recent advances on
amenable algebras as well as LEF algebras and algebras having almost finite dimensional representations \cite{VG:lef,Zi,El:aa,El,Gr08}.
The idea allows to go beyond algebras associated to groups. In such a general context, we introduce
the concepts of \emph{linear sofic radical}  for groups and of \emph{sofic radical} for algebras.
We also notice the existence of algebras which are not linear sofic (these are not group algebras).
We refer to Section~\ref{sec:almostfd} for details.

Our choice of the rank metric is not arbitrary but has a view towards potential applications.
The concept of the rank metric was first introduced by Loo-Keng Hua \cite{Hua:45}  (he uses the term ``arithmetic distance'')
who found a surprisingly nice description of adjacency preserving maps with respect to this metric. The entire book \cite{Wan}
is devoted to this topic, see also \cite{Semrl} for a recent discussion on the connection to
several preserver problems on matrix and operator algebras arising in physics and geometry.
From a different point of view, Philippe Delsarte \cite{Del} defined the rank distance (named $q$-distance) on the set of bilinear forms
and proposed the construction of optimal codes in bilinear form representation. This allowed Ernst Gabidulin \cite{Gab} to study
the rank distance for vector spaces over extension fields and to describe optimal codes, now called Gabidulin codes. This currently emerged 
into an intensively developing area of rank-metric codes. 

 In view of the above, we believe that the full extent of
possible applications of sofic and linear sofic groups and algebras is yet to be discovered. The present paper 
provides the necessary fundaments for such further developments.

\section{Ultraproducts of matrix algebras with respect to the
rank}

Ultraproducts of matrices using rank functions have been
considered, for example, in \cite{El-Sz1,Oz}. Let us first recall some basic properties of the
rank. Throughout the article $F$ is an arbitrary field.

\begin{nt}
For a matrix $a\in M_n=M_n(F)$ we shall denote by $rk(a)$ its rank
and define the normalized rank by $\rho(a):=\frac1nrk(a)$.
\end{nt}

\begin{p}\label{prop of rk}
The rank function on complex matrices has the following
properties:
\begin{enumerate}
\item $rk(I_n)=n;$ $rk(a)=0$ if and only if $a=0;$
\item $rk(u+v)\leqslant
rk(u)+rk(v)$ for $u,v\in M_n;$
\item $rk(uv)\leqslant rk(u)$ and $
rk(uv)\leqslant rk(v)$ for $u,v\in M_n;$
 \item $rk(u\oplus
v)=rk(u)+rk(v)$ for $u\in M_n$ and $v\in M_m;$ \item $rk(u\otimes
v)=rk(u)\cdot rk(v)$ for $u\in M_n$ and $v\in M_m.$
\end{enumerate}
\end{p}

We now define the ultraproduct that we use throughout the paper.

\begin{de}\label{rankultraproduct}
Let $\omega$ be a non-principal (or free) ultrafilter and $(n_k)_k$ a sequence of
natural numbers such that $\lim_{k\to\infty} n_k=\infty$. The
Cartesian product $\Pi M_{n_k}(F)$ is an algebra. Let us define:
\[\rho_\omega:\Pi M_{n_k}(F)\to[0,1]\ \ \
\rho_\omega((a_k)_k):=\lim_{k\to\omega}\rho(a_k).\] Then $Ker
\rho_\omega$ is an ideal of $\Pi M_{n_k}(F)$. We denote by
$\Pi_{k\to\omega} M_{n_k}(F)/Ker\rho_\omega$, or  by $\Pi
M_{n_k}(F)/Ker\rho_\omega$ if there is no danger of confusion, the
ultraproduct obtained by taking the quotient of $\Pi M_{n_k}(F)$ by $Ker
\rho_\omega$. This algebra comes with a natural metric defined by $d_\omega(a,b):=\rho_\omega(a-b)$,
where $a$ and $b$ belong to the ultraproduct.
\end{de}

We always denote by $\rho_\omega$ the limit rank function, even
though we shall work with ultraproducts over different dimension
sequences $(n_k)_k$. It will be clear what is the dimension of the
matrices that we use, so this notation should cause no confusion.

\begin{ob}
The function $\rho$ induces a metric on $M_n(F)$, defined by
$d_{rk}(a,b):=\rho(a-b)$. This metric restricted to the group
$GL_n(F)$ is bi-invariant. Thus, we can construct the following
ultraproduct.
\end{ob}

\begin{de}\label{universallinearsofic}
We denote by $\Pi_{k\to\omega} GL_{n_k}(F)/d_\omega$ the metric
ultraproduct obtained by taking the quotient of the Cartesian product $\Pi
GL_{n_k}(F)$ by $\nr_\omega=\{(a_k)_k\in\Pi
GL_{n_k}(F):\lim_{k\to\omega}d_{rk}(a_k,Id)=0\}$.
\end{de}

In many ultraproduct constructions invertible elements in an
ultraproduct are given by ultraproduct of invertible elements.
For instance, this is the case in the classical result of Malcev addressing the
algebraic ultraproduct of matrix algebras \cite{Ma}. An analogous result
holds also for our rank ultraproduct construction.

\begin{p}\label{invertible in ultraproduct}
The group $\unit(\Pi_{k\to\omega} M_{n_k}(F)/Ker\rho_\omega)$ of
invertible elements of $\Pi_{k\to\omega}
M_{n_k}(F)/Ker\rho_\omega$ is isomorphic to $\Pi_{k\to\omega}
GL_{n_k}(F)/d_\omega$.
\end{p}
\begin{proof}
Elements of $\Pi_{k\to\omega} GL_{n_k}(F)/d_\omega$ are
invertible. So, $\Pi_{k\to\omega}
GL_{n_k}(F)/d_\omega\subseteq\unit(\Pi_{k\to\omega}
M_{n_k}(F)/Ker\rho_\omega)$. For the converse inclusion, the key
observation is that for any $a\in M_{n_k}(F)$ there exists $\tilde
a\in GL_{n_k}(F)$ such that $\rho(a-\tilde a)=1-\rho(a)$.

If $(a_k)_{k,\omega}$ is invertible in $\Pi_{k\to\omega}
M_{n_k}(F)/Ker\rho_\omega$ then $\rho_\omega((a_k)_{k,\omega})=1$
by (1) and (3) of Proposition \ref{prop of rk} (also the reverse
of this implication holds). It follows that
$(a_k)_{k,\omega}=(\tilde a_k)_{k,\omega}\in\Pi_{k\to\omega}
GL_{n_k}(F)/d_\omega$.
\end{proof}

We shall encounter many examples of stably finite algebras. Let us
first recall the definition.

\begin{de}
A unital ring $R$ is called \emph{directly finite} if for any
$x,y\in R$, $xy=I$ implies $yx=I$. It is called \emph{stably
finite} if $M_n(R)$ is directly finite for any $n\in\nz$.
\end{de}

\begin{Kap}
For any field $F$ and any countable group $G$ the group algebra
$F(G)$ is directly finite.
\end{Kap}

The following proposition is well known. It was used by Elek and
Szabo to prove Kaplansky's direct finiteness conjecture for sofic
groups and it also appears in \cite{Oz}. Note that the class of sofic goups is
currently the largest known to satisfy this conjecture.

\begin{p}\label{stably finite}
The algebra $\Pi_{k\to\omega} M_{n_k}(F)/Ker\rho_\omega$ is stably
finite.
\end{p}
\begin{proof}
As $M_m(\Pi_{k\to\omega} M_{n_k}(F)/Ker\rho_\omega)
\simeq\Pi_{k\to\omega} M_{m\cdot n_k}(F)/Ker\rho_\omega$ we only
need to prove the direct finiteness of these algebras.

It is not hard to check that $rk(I-ab)=rk(I-ba)$ for $a,b\in
M_n(F)$. This equality implies in the ultralimit that
$\rho_\omega(I-xy)=\rho_\omega(I-yx)$. So, in $\Pi_{k\to\omega}
M_{n_k}(F)/Ker\rho_\omega$ we have $xy=I$ if and only if $yx=I$.
\end{proof}

\section{Product of ultrafilters}

We have equipped the ultraproduct $\Pi_{k\to\omega} M_{n_k}(F)/Ker\rho_\omega$
with a metric induced by the rank function, namely $d_\omega(a,b):=\rho_\omega(a-b)$. If
we have a family of ultraproducts we can construct the metric
ultraproduct of this family. The object that we get is again an
ultraproduct. We provide here the definitions. For more details, we refer
the reader to~\cite{Ca-Pa}.

\begin{de}
If $\phi,\omega$ are ultrafilters on $\nz$, then define the
\emph{product ultrafilter} $\phi\otimes\omega$ on $\nz\times\nz$
by:
\[A\in\phi\otimes\omega\Longleftrightarrow\{i\in\nz:\{j\in\nz:(i,j)\in
A\}\in\omega\}\in\phi.\]
\end{de}

It is easy to check that $\phi\otimes\omega$ is an ultrafilter. Since
$\nz$ and $\nz^2$ are cardinal equivalent, $\phi\otimes\omega$ can
be viewed as an ultrafilter on $\nz$.

\begin{p}\label{twolimits}
If $(x_i^j)_{(i,j)\in\nz^2}$ is a bounded sequence of real
numbers then:
\[\lim_{i\to\phi}(\lim_{j\to\omega}x_i^j)=\lim_
{(i,j)\to\phi\otimes\omega}x_i^j.\]
\end{p}

This proposition implies that the ultraproduct of ultraproducts is
again an ultraproduct:

\begin{cor}
Let $(n_{m,k})_{m,k}$ be a double sequence of natural numbers.
For every $m\in\nz$ construct the ultraproduct
$\Pi_{k\to\omega}M_{n_{m,k}}/Ker\rho_\omega$. On the Cartesian
product $\Pi_m(\Pi_{k\to\omega}M_{n_{k,m}})$ define
$d_\phi((a_m)_m,(b_m)_m)=\lim_{m\to\phi}\rho_\omega(a_m-b_m)$.
Then:
\[\Pi_{m\to\phi}(\Pi_{k\to\omega}M_{n_{m,k}}/Ker\rho_\omega)/d_\phi
\simeq\Pi_{(m,k)\to\phi\otimes\omega}M_{n_{m,k}}/Ker\rho_{\phi\otimes\omega}.\]
\end{cor}

\section{Definitions of linear soficity}\label{sec:def}

We are now defining the main concepts of our paper.

\begin{de}\label{def:lsofic}
A countable group $G$ is \emph{linear sofic} if there exist an
injective morphism $\Theta:G\to\Pi_{k\to\omega}
GL_{n_k}(\cz)/d_\omega$.

Such a morphism is called a \emph{linear sofic representation} of
$G$.
\end{de}

\begin{de}\label{deflinearsofic}
A countably generated algebra $A$ over a field $F$ is \emph{linear
sofic} if there exist an injective morphism $\Theta:A\to\Pi_{k\to\omega}
M_{n_k}(F)/Ker\rho_\omega$. Moreover, if $A$ is a unital algebra
we require that this morphism is unital.

Such a morphism is called a \emph{linear sofic representation} of
$A$.
\end{de}

\begin{ob}\label{vectorspaces}
An element of $M_n(\cz)$ is a linear transformation of the vector
space $\cz^n$. As $\cz$ is a vector space of dimension $2$ over
$\rz$, we can view this element as a transformation of the space
$\rz^{2n}$ or as a matrix in $M_{2n}(\rz)$. Its normalized rank
remains the same. As a consequence, a morphism $\Theta:G\to
\Pi_{k\to\omega} GL_{n_k}(\cz)/Ker\rho_\omega$ induces a morphism
$\Theta':G\to\Pi_{k\to\omega} GL_{2n_k}(\rz)/Ker\rho_\omega$. The
value of $\rho_\omega$ is preserved by this transformation. It
follows that we can work with $\rz$ instead of $\cz$ in the
definition of linear sofic group. We can further reduce
our considerations to the field of rationals (or equivalently to any finite dimensional
extension of the rationals), see Section~\ref{sec:q}. Alternatively, we construct linear sofic representations of groups over
a sequence of finite fields, see Section \ref{sec:wsofic}.
\end{ob}

Let $G$ be a linear sofic group and $\Theta:G\to\Pi_{k\to\omega}GL_{n_k}/d_\omega$ be an injective morphism. We define a length
function $\delta:G\to[0,1]$ by $\delta(g)=d_\omega(1,\Theta(g))$.
Then $\delta$ is constant on conjugacy classes of $G$ and
$\delta(g)=0$ iff $g=e$. The following proposition is
straightforward. It gives a more algebraic definition of linear soficity.
We will provide a stronger version we alluded to in the introduction, see Section \ref{sec:rampl}.

\begin{p}\label{algebraic definition}
A group $G$ is linear sofic if and only if the following holds: there exists $\delta:G\to[0,1]$
such that $\delta(g)=0$ if and only if $g=e$ and for any finite
subset $E\subset G$ and  for any $\vp>0,$ there exist $n\in\nz$ and a function
$\phi:E\to GL_n(\cz)$ such that:
\begin{enumerate}
\item $\forall g,h,gh\in E$ we have
$\rho(\phi(g)\phi(h)-\phi(gh))<\vp;$ \item $\forall g\in E$ we
have $\rho(1-\phi(g))>\delta(g)-\vp.$
\end{enumerate}
\end{p}

Using this equivalent characterization we see that the definition of
linear sofic group does not depend on the particular choice of the
ultrafilter nor does it depend on the sequence $(n_k)_k$ as long
as $\lim_{k\to \infty}n_k=\infty$.

It is implicit in \cite{El-Sz1} that for a sofic group $G$ and any
field $F$ the group algebra $F(G)$ satisfies our definition of
linear sofic algebra (see also \cite{Oz}). This was obtained as an
intermediate result in the proof of Kaplansky's direct finiteness
conjecture for sofic groups. We present an alternative proof
of this fact in Section \ref{sec:lin}, as a consequence of our results about linear
soficity. The following result shows that linear soficity is a priori weaker than soficity.
Observe that the converse is open.

\begin{p}\label{sofic implies lsofic}
Sofic groups are linear sofic.
\end{p}
\begin{proof}Let $p\in S_n$ and let $A_p$ be the corresponding permutation
matrix. Denote by $fix(p)$ the number of fixed points of $p$ and
by $cyc(p)$ the number of cycles (including fixed points) of $p$.
Clearly, $fix(p)\leqslant cyc(p)$. By definition,
$d_{Hamm}(Id,p)=1-fix(p)/n$ and it is easy to check that
$\rho(Id-A_p)=1-cyc(p)/n$ (see \cite{Lu}, Lemma 13). From this we
deduce that $\rho(Id-A_p)\leqslant d_{Hamm}(Id,p)$.

Let $G$ be a sofic group and $\Phi:G\to\Pi_{k\to\omega}
(S_{n_k},d_{Hamm})$ the corresponding injective morphism. The group
$S_{n_k}$ is isomorphic to the subgroup of permutation matrices of
$GL_{n_k}(\cz)$. Due to the above inequality on the normalized rank,
the morphism $\Phi$ induces a group morphism $\Theta:G\to
\Pi_{k\to\omega} GL_{n_k}(\cz)/d_\omega$. We show that this
morphism is injective.

Coming back to $p$ and $A_p$, it is easy to see that
$cyc(p)\leqslant fix(p)+(n-fix(p))/2$. This can be rewritten as
$1-fix(p)/n\leqslant 2(1-cyc(p)/n)$. Thus,
\[d_{Hamm}(Id,p)\leqslant 2\rho(Id-A_p).\]
As a consequence, we deduce the injectivity of $\Theta$.
\end{proof}

\section{Rank amplification}\label{sec:rampl}

A classical theorem of Elek and Szabo states that if $G$ is a sofic
group then there exists a group morphism
$\Theta\colon G\to\Pi_{k\to\omega}(S_{n_k},d_{Hamm})$ such that the
distance between $\Theta(g_1)$ and $\Theta(g_2)$ is 1 in the limit
for each $g_1\neq g_2$. This fact is required to
prove various results  including some permanence properties like  a \emph{direct limit of sofic groups is again
sofic}. We shall obtain a similar general fact for linear sofic groups. That is, in
Proposition~\ref{algebraic definition} we make a function $\delta$
constant on $G\setminus\{e\}$:  $\delta$ is independent of the choice of the group element $g\in G\setminus\{e\}$.
The proof of Elek and Szabo employs
a fundamental tool called \emph{amplification}. In matrix language, this
tool relies on tensor products together with the formula
$Tr(a\otimes b)=Tr(a)Tr(b),$ where $Tr$ denotes  the trace of a
matrix. Unfortunately, we do not have a similar formula for the rank metric.
Thus, our proof is technically much more involved.

\subsection{Preliminaries}

In this section, $A$ is always an element of $GL_n(\cz)$. For
$\lambda\in\cz$ define $M_\lambda(A)$ to be $1/n$ multiplied with
the algebraic multiplicity of the eigenvalue $\lambda$ (this is $0$ whenever $\lambda$ is not an eigenvalue of $A$). Then $M_0(A)=0$
and $\sum_{\lambda\in\cz^*}M_\lambda(A)=1$. Observe that:
\[\rho(A-Id)\geqslant 1-M_1(A).\]

\begin{lemma}
If $(\lambda_i)_{i=1,\ldots,n}$ are the eigenvalues of $A$ written
with the algebraic multiplicity, then
$(\lambda_i\lambda_j)_{i,j=1,\ldots,n}$ are the eigenvalues of
$A\otimes A$ written with algebraic multiplicity.
\end{lemma}
\begin{proof}
By the Jordan decomposition, we can write $A$ as an upper
triangular matrix with the values $(\lambda_i)_{i=1,\ldots,n}$ on
the diagonal. Then $A\otimes A$ is also an upper triangular matrix
with the values $(\lambda_i\lambda_j)_{i,j=1,\ldots,n}$ on the
diagonal. This implies that
$(\lambda_i\lambda_j)_{i,j=1,\ldots,n}$ are the roots of the
characteristic polynomial of $A\otimes A$. These roots are the
eigenvalues with algebraic multiplicity.
\end{proof}

\begin{lemma}\label{twosquaresineq}
If $A\in GL_n(\cz),$ then $M_1(A\otimes A)\leqslant
M_1(A)^2+(1-M_1(A))^2$.
\end{lemma}
\begin{proof}
Let $(\lambda_i)_{i=1,\ldots,n}$ be the eigenvalues of $A$ written
with the algebraic multiplicity. Assume that for $i=1,\ldots,k$ we
have $\lambda_i=1$ and for $i=k+1,\ldots,n$ we have $\lambda_i\neq
1$. Then $M_1(A)=k/n$. If $\lambda_i\lambda_j=1$ then either
$i\leqslant k$ and $j\leqslant k$ or $i>k$ and $j>k$. This implies
that $M_1(A\otimes A)\leqslant
(k^2+(n-k)^2)/n^2=M_1(A)^2+(1-M_1(A))^2$.
\end{proof}

The following proposition is elementary.

\begin{p}
Define $f:[1/2,1]\to[1/2,1]$ by $f(x)=x^2+(1-x)^2$. Then $f$ is a
well-defined increasing bijection. If $x\in[1/2,1),$ then
$\lim_{m\to\infty}f^m(x)=1/2$.
\end{p}

\begin{lemma}
If there exists $\mu\in\cz$ such that $M_\mu(A\otimes A)>1/2,$ then
there exists $\lambda\in\cz$ such that $M_\lambda(A)>1/2$.
Moreover, both $\lambda$ and $\mu$ are unique and if $\lambda=1$
then $\mu=1$.
\end{lemma}
\begin{proof}
Let $(\lambda_i)_{i=1,\ldots,n}$ be the eigenvalues of $A$ written
with the algebraic multiplicity. By hypothesis, there exists
$\mu\in\cz$ such that
$|\{(i,j)\colon\lambda_i\lambda_j=\mu\}|>\frac12n^2$. Let
$C_i=\{j\colon\lambda_i\lambda_j=\mu\}$. Then
$\sum_{i=1}^n|C_i|>\frac12n^2$. It follows that there exists $i_0$
such that $|C_{i_0}|>\frac12n$. If $\lambda=\lambda_{i_0}^{-1}\mu$
then $\lambda_j=\lambda$ for every $j\in C_{i_0}$. This means that
$M_\lambda(A)>1/2$. The uniqueness part of the proposition is
trivial.

Suppose now that $\lambda=1$ and assume that $\mu\neq 1$. Let
$k_1=M_1(A)\cdot n$ and $k_2=M_\mu(A)\cdot n$ (these are the algebraic
multiplicities of $1$ and $\mu$). Define $k_3=n-k_1-k_2$. By
hypothesis $k_1>n/2,$ hence $k_3<k_1$. It is easy to see that the
algebraic multiplicity of $\mu$ in $A\otimes A$ is less than
$2k_1k_2+k_3^2$. Then:
\[2n^2M_\mu(A\otimes A)\leqslant4k_1k_2+2k_3^2\leqslant(k_1+k_2)^2+
(k_1+k_2)k_3+k_3^2\leqslant(k_1+k_2+k_3)^2=n^2.\] It follows that
$M_\mu(A\otimes A)\leqslant 1/2$, giving a contradiction.
\end{proof}

\begin{de}
For $A\in GL_n(\cz)$ define $A_1:=A$ and $A_{m+1}:=A_m\otimes A_m$.
\end{de}

\begin{p}\label{technical result}
Suppose that $M_1(A)\in(1/2,1)$ and let $c$ be a constant such
that $c\in(M_1(A),1)$. Then $M_1(A_m)<f^{m-1}(c)$.
\end{p}
\begin{proof}
We proceed by induction. The case $m=1$ follows by hypothesis.
Assume that $M_1(A_m)<f^{m-1}(c)$. Since $f^m(c)$ is always
strictly greater than $1/2$, if $M_1(A_{m+1})\leqslant 1/2$ we are
done. Assume that $M_1(A_{m+1})>1/2$. We shall prove that also
$M_1(A_m)>1/2$.

By the previous lemma and a reverse induction, for each
$n\leqslant m$ there exists $\lambda_n$ such that
$M_{\lambda_n}(A_n)>1/2$. Since $M_1(A)>1/2$ we get
$\lambda_1=1$. Applying the second part of the previous lemma, we
inductively get $\lambda_n=1$ for $n\leqslant m$. This proves that
$M_1(A_m)>1/2$.

The function $f$ is strictly increasing, therefore $f(M_1(A_m))<f^m(c)$.
By Lemma \ref{twosquaresineq}, we get $M_1(A_{m+1})\leqslant
f(M_1(A_m))$ and we are done.
\end{proof}

This proposition solves the case $M_1(A)<1$. We still
have to deal with the case $M_1(A)=1$, that is when all eigenvalues of
$A$ are $1$. In this case, the inequality $\rho(A-Id)\geqslant
1-M_1(A)$ cannot help. Thus, we have to investigate in detail the decomposition
of $A$ into Jordan blocks.

Let $J(\alpha,s)\in GL_s(\cz)$ be the Jordan block of size
$s\times s$ and having eigenvalue $\alpha$, that is, the diagonal is composed
only of values $\alpha$ and the entries directly above and to the
right of the diagonal are $1$. We use the following recent
description of the tensor product of Jordan blocks (surprisingly, the proof of this fact
is quite involved).

\begin{te}(\cite[Theorem 2]{Ma-Vl}, \cite[Theorem 2.0.1]{Ii-Iw})
For $s,t\in\nz$, $s\leqslant t$ and $\alpha,\beta\in\cz^*$ we
have:
\[J(\alpha,s)\otimes J(\beta,t)=\oplus_{i=1}^s
J(\alpha\beta,s+t+1-2i).\]
\end{te}

From now on, $A$ is a matrix in $GL_n(\cz)$ such that $M_1(A)=1$.
Denote by $J(A)$ the number of Jordan blocks in $A$ divided by
$n$. Then $\rho(A-Id)=1-J(A)$.

\begin{p}
If $M_1(A)=1$ then $J(A\otimes A)\leqslant J(A)$ and $J(A\otimes
A)\leqslant J(A)^2+(1-J(A))^2$.
\end{p}
\begin{proof}
For $i\in\nz^*$ let $c_i$ be the number of Jordan blocks in $A$ of
size $i$. Then $\sum_iic_i=n$ and $\sum_ic_i=nJ(A)$. Also
$A=\oplus_iJ(1,i)\otimes Id_{c_i}$. Then:
\[A\otimes A=\oplus_{i,j}J(1,i)\otimes J(1,j)\otimes
Id_{c_ic_j}.\] According to the previous theorem the number of
Jordan blocks in the matrix $J(1,i)\otimes J(1,j)$ is
$\min\{i,j\}$, so $n^2J(A\otimes A)=\sum_{i,j}c_ic_j\min\{i,j\}$.
Then:
\[n^2J(A\otimes
A)=\sum_{i<j}c_ic_ji+\sum_{j<i}c_ic_jj+\sum_ic_i^2i=\sum_iic_i^2+2\sum_{i<j}ic_ic_j.\]
Note that $n^2[J(A)^2+(1-J(A))^2]=(nJ(A))^2+(n-nJ(A))^2$, so:
\begin{align*}
n^2(J(A)^2+(1-J(A))^2)=&(\sum_ic_i)^2+(\sum_i(i-1)c_i)^2=\sum_ic_i^2+2\sum_{i<j}c_ic_j+\\
+&\sum_i(i-1)^2c_i^2+2\sum_{i<j}(i-1)(j-1)c_ic_j\\
=&\sum_i[(i-1)^2+1]c_i^2+2\sum_{i<j}[(i-1)(j-1)+1]c_ic_j.
\end{align*}
For the second inequality we only need to see that for any
$i\in\nz$ and $i<j$ we have $i\leqslant (i-1)^2+1\leqslant
(i-1)(j-1)+1$. The first inequality is easy because:
\[n^2J(A)=(\sum_iic_i)(\sum_jc_j)=\sum_{i,j}ic_ic_j.\]
\end{proof}

\begin{p}\label{M1=1 technical result}
Let $c>1/2$ be a constant $c\in(J(A),1)$. Then
$J(A_m)<f^{m-1}(c)$.
\end{p}
\begin{proof}
For $m=1$, $A_m=A$ and $f^0(c)=c,$ hence the result follows by
hypothesis. Suppose now that $J(A_m)<f^{m-1}(c)$. If
$J(A_{m})\leqslant 1/2,$ then $J(A_{m+1})\leqslant 1/2$ and we are
done as $f^m(c)$ is always strictly greater than $1/2$.

Assume that $J(A_{m+1})>1/2$. The function $f$ is strictly
increasing, so $f(J(A_m))<f^m(c)$. By the previous proposition,
$J(A_{m+1})\leqslant f(J(A_m))$ and we are done.
\end{proof}

\subsection{Equivalent definition}

This section is devoted to the proof of the following theorem which will provide
the strengthening of Proposition \ref{algebraic definition}.

\begin{te}
Let $G$ be a countable linear sofic group. Then there exists a
morphism $\Psi\colon G\to\Pi_{k\to\omega} GL_{n_k}(\cz)/d_\omega$ such
that $d_\omega(\Psi(g),Id)\geqslant\frac14$ for any $g\neq e$.
\end{te}
\begin{proof}
Let $\Theta\colon G\to\Pi_{k\to\omega} GL_{n_k}(\cz)/d_\omega$ be a
linear sofic representation of $G$. Let $\theta^k(g)\in GL_{n_k}$
be such that $\Theta(g)=\Pi_{k\to\omega}\theta^k(g)/d_\omega$.
Define $\theta_1^k(g):=\theta^k(g)$ and
$\theta_{m+1}^k(g):=\theta_m^k(g)\otimes\theta_m^k(g)$. Notice that
the matrix dimension of $\theta_m^k(g)$ is $n_k^{2^{m-1}}.$

Construct the linear sofic representation:
\[\Theta_m\colon G\to\Pi_{k\to\omega} GL_{n_k^{2^{m-1}}}(\cz)/d_\omega, \
\ \ \Theta_m(g)=\Pi_{k\to\omega}\theta_m^k(g)/d_\omega,\] and take
the ultraproduct of these representations:
\[\Psi_1\colon G\to\Pi_{(m,k)\to\omega\otimes\omega}
GL_{n_k^{2^{m-1}}}(\cz)/d_{\omega\otimes\omega}, \hbox{ where }\] 
\[\Psi_1(g)=\Pi_{m\to\omega}\Theta_m(g)/d_\omega=
\Pi_{(m,k)\to\omega\otimes\omega}\theta_m^k(g)/d_{\omega\otimes\omega}.\]
Also construct an amplification of $\Theta$ to this sequence of
matrix dimensions:
\[\Psi_2\colon G\to\Pi_{(m,k)\to\omega\otimes\omega}
GL_{n_k^{2^{m-1}}}(\cz)/d_{\omega\otimes\omega}, \hbox{ where }\] 
\[\Psi_2(g)=\Pi_{(m,k)\to\omega\otimes\omega}\theta^k(g)\otimes
Id_{n_k^{2^{m-1}-1}}/d_{\omega\otimes\omega}.\]

Define $\Psi=\Psi_1\oplus\Psi_2$,
$\Psi\colon G\to\Pi_{(m,k)\to\omega\otimes\omega}
GL_{2n_k^{2^{m-1}}}(\cz)/d_{\omega\otimes\omega}$ such that:
\[\rho_{\omega\otimes\omega}(\Psi(g)-Id)=\frac12\big(\rho_{\omega\otimes\omega}
(\Psi_1(g)-Id)+\rho_{\omega\otimes\omega}(\Psi_2(g)- Id)\big).\]

\begin{claim}
For any $g\in G,$ we have
$\rho_{\omega\otimes\omega}(\Psi(g)-Id)\geqslant 1/4$.
\end{claim}

Assume that $\lim_{k\to\omega} M_1(\theta^k(g))\leqslant1/2$. Then
$\lim_{n\to\omega}\lim_{k\to\omega} M_1(\theta^k(g)\otimes
Id)\leqslant1/2$. It follows that
$\rho_{\omega\otimes\omega}(\Psi_2(g)-Id)\geqslant1/2,$ hence
$\rho_{\omega\otimes\omega}(\Psi(g)-Id)\geqslant 1/4$ and we are
done. We are left with the case $\lim_{k\to\omega}
M_1(\theta^k(g))>1/2$.

Assume that $\lim_{k\to\omega} M_1(\theta^k(g))<1$. Then there
exist $c\in(1/2,1)$ and $F\in\omega$ such that
$1/2<M_1(\theta^k(g))<c$ for all $k\in F$. It follows by Lemma
\ref{technical result} that:
\[M_1(\theta_m^k(g))<f^{m-1}(c),\ \ \forall k\in F,\ m\in\nz.\]
Then $\lim_{k\to\omega}M_1(\theta_m^k(g))\leqslant f^{m-1}(c)$ for
all $m\in\nz$. We get the inequality:
\[\lim_{m\to\omega}\lim_{k\to\omega}M_1
(\theta_m^k(g))\leqslant\lim_{m\to\omega}f^{m-1}(c)=1/2.\] As a
consequence $\rho_{\omega\otimes\omega}(\Psi_1(g)-Id)\geqslant1/2$,
so we have $\rho_{\omega\otimes\omega}(\Psi(g)-Id)\geqslant 1/4$.

We are left with the case $\lim_{k\to\omega} M_1(\theta^k(g))=1$.
We can assume that $M_1(\theta^k(g))=1$ for each $k$. It follows
that $M_1(\theta_m^k(g))=1$ for each $m$ and $k$. The proof is
similar to the previous case, using $J(\theta_m^k(g))$ instead of
$M_1(\theta_m^k(g))$, the equation $\rho(A)=1-J(A),$ and
Proposition~\ref{M1=1 technical result} instead of Proposition~\ref{technical result}.
\end{proof}

The structure of the group does not play a role in the proof as
$\Theta(g_1)$ does not interact with $\Theta(g_2)$ for $g_1\neq
g_2$. The construction is possible even if we have just a subset
of our group.

\begin{p}\label{large distance}
Let $G$ be a countable group and let $E$ be a subset of $G$.
Consider a
function $\Phi\colon E\to\Pi_{k\to\omega}GL_{n_k}(\cz)/d_\omega$  such that $\Phi(g)\Phi(h)=\Phi(gh)$ whenever $g,h,gh\in
E$. Then there exists
$\Psi\colon E\to\Pi_{k\to\omega}GL_{m_k}(\cz)/d_\omega$ such that
$\Psi(g)\Psi(h)=\Psi(gh)$ whenever {$g,h,gh\in E$} and:
\begin{align*}
&\Phi(g)=\Phi(h)\Longrightarrow\Psi(g)=\Psi(h)\\
&\Phi(g)\neq\Phi(h)\Longrightarrow
d_\omega(\Psi(g),\Psi(h))\geqslant\frac14.
\end{align*}
\end{p}

Now we can provide a stronger version of the algebraic characterization
of linear soficity contained in Proposition \ref{algebraic definition} using
this extra information that we obtained.

\begin{p}\label{16 algebraic definition}
A group $G$ is linear sofic if and only if for any  finite subset $E\subset G$ and
for any $\vp>0$ there exists $n\in\nz$ and a function $\phi\colon E\to
GL_n(\cz)$ such that:
\begin{enumerate}
\item $\forall g,h,gh\in E$ we have
$\rho(\phi(g)\phi(h)-\phi(gh))<\vp;$ \item $\forall g\in E$ we
have $\rho(1-\phi(g))>\frac14-\vp.$
\end{enumerate}
\end{p}

\section{Rational linear soficity}\label{sec:q}

This section is devoted to proving that in the definition of
linear sofic group (see Definition~\ref{def:lsofic}) we can use the groups $GL_n(\qz)$ endowed with
the rank metric. In other words, the existence of a complex linear sofic representation is equivalent to
the existence of a rational linear sofic representation.

\begin{lemma}\label{lsoficwithbase}
A group $G$ is linear sofic if and only if for any  finite subset
$E\subset G$ and for any $\vp>0$ there exist $n\in\nz$ and a
function $\phi\colon E\to GL_n(\rz)$ such that:
\begin{enumerate}
\item $\forall g,h,gh\in E$ at least $(1-\vp)n$ columns of the matrix $\phi(g)\phi(h)$
are equal to the corresponding columns in $\phi(gh);$
 \item $\forall g\in E$ we
have $\rho(1-\phi(g))>\frac14-\vp.$
\end{enumerate}
\end{lemma}
\begin{proof}
Elements in $GL_n(\rz)$ are invertible linear transformations on
$\rz^n$. These elements are matrices as soon as we fix a basis for
the vector space $\rz^n$. As the second condition does not depend
on the particular choice of a basis, we only need to concentrate
our efforts to constructing a basis such that the first condition
holds.

Let $E_1=\{(g,h)\in E^2\mid gh\in E\}$. Then $E_1$ is a finite set.
Take $\delta:=\vp/|E_1|$. Apply Proposition \ref{16 algebraic
definition} for $E$ and $\delta$ to get $n\in\nz$ and  a function
$\phi\colon E\to GL_n(\rz)$ (use Observation \ref{vectorspaces} to
replace $\cz$ by $\rz$). For each $(g,h)\in E_1$ let
$V_{g,h}\subset\rz^n$ be the linear subspace on which
$\phi(g)\phi(h)=\phi(gh)$. By condition $(1)$ of Proposition
\ref{16 algebraic definition} it follows that $\dim V_{g,h}>
(1-\delta)n$ for any $(g,h)\in E_1$. Let $V=\bigcap_{(g,h)\in
E_1}V_{g,h}$ . Then $\dim V>(1-|E_1|\delta)n=(1-\vp)n$. Choose a
basis in $V$ and complete it to a basis in $\rz^n$.

Using this basis we can see elements in $\phi(E)$ as matrices. It
is clear now by construction that the first condition holds.
\end{proof}

We denote by $\overline\qz$ the field of real algebraic numbers.
The next step in the proof is to replace the function $\phi\colon E\to
GL_n(\rz)$ by another function $\psi\colon E\to GL_n(\overline\qz)$. In order to achieve this we will use the following variant of the fundamental
result of real semialgebraic geometry, so-called
\emph{Positivstellensatz}.

\begin{te}(\cite[Theorem 4.4.2, p. 92]{rag})
Let $R$ be a real closed field. Let $(f_j)_{j=1,\ldots,s},
(g_k)_{k=1,\ldots,t},$ and $(h_l)_{l=1,\ldots,u}$ be finite
families of polynomials in $R[X_1,\ldots,X_d]$. Denote by $P$ the
cone generated by $(f_j)_{j=1,\ldots,s}$, by $M$ the multiplicative
monoid generated by $(g_k)_{k=1,\ldots,t},$ and by $I$ the ideal
generated by $(h_l)_{l=1,\ldots,u}$. Then the following properties
are equivalent:
\begin{enumerate}
\item The set $\{x\in R^d\mid f_j(x)\geqslant 0\ \forall j,\ 
g_k(x)\neq 0\ \forall k,\ h_l(x)=0\ \forall l\}$ is empty.\\
\item There exist $f\in P$, $g\in M,$ and $h\in I$ such that
$f+g^2+h=0$.
\end{enumerate}
\end{te}

\begin{cor}
Let $(g_k)_{k=1,\ldots,t}$ and $(h_l)_{l=1,\ldots,u}$ be finite
families of polynomials in $\qz[X_1,\ldots,X_d]$. If there exists
a real solution $x\in\rz^d$ to the system:
\begin{align*}
&g_k(x)\neq 0\ &k=1,\ldots,t, \\ &h_l(x)=0\ &l=1,\ldots,u,
\end{align*}
then there is also a solution $x\in\overline\qz^d$.
\end{cor}
\begin{proof}
Let $P$ be the smallest cone in $\overline\qz[X_1,\ldots,X_d]$,
that is $P$ contains squares in $\overline\qz[X_1,\ldots,X_d]$ and it
is closed under addition and multiplication by positive scalars.
Let also $M$ be the multiplicative monoid generated by
$(g_k)_{k=1,\ldots,t}$ and $I$ the ideal generated by
$(h_l)_{l=1,\ldots,u}$ in $\overline\qz[X_1,\ldots,X_d]$.

If there is no solution $x\in\overline\qz^d$ to the system above, then
according to the previous theorem there exist $f\in P$, $g\in M,$
and $h\in I$ such that $f+g^2+h=0$. However, this equation also holds in
$\rz[X_1,\ldots,X_d]$ so there should not exist a solution
$x\in\rz^d$.
\end{proof}

\begin{p}
A group $G$ is linear sofic if and only if for any  finite subset $E\subset G$ and
for any $\vp>0$ there exist $n\in\nz$ and a function $\phi\colon E\to
GL_n(\overline\qz)$ such that:
\begin{enumerate}
\item $\forall g,h,gh\in E$ we have
$\rho(\phi(g)\phi(h)-\phi(gh))<\vp;$ \item $\forall g\in E$ we
have $\rho(1-\phi(g))>\frac14-\vp.$
\end{enumerate}
\end{p}
\begin{proof}
Using $E$ and $\vp$, apply Lemma \ref{lsoficwithbase} to get a
function $\phi\colon E\to GL_n(\rz)$. Recall from the proof of that
lemma that $E_1=\{(g,h)\in E^2\mid gh\in E\}.$

We regard conditions $(1)$ and $(2)$ of Lemma
\ref{lsoficwithbase} as a system of equations and non-equations.
The variables of this system are the $n^2|E|$ entries of matrices
in $\phi(E)$.

Condition $(1)$ in Lemma \ref{lsoficwithbase} provides more than
$(1-\vp)n$ equations for each pair $(g,h)\in E_1$. These equations
are enough to deduce that $\rho(\phi(g)\phi(h)-\phi(gh))<\vp$. For
each $g\in E$ choose a minorant of $1-\phi(g)$ of size greater
than $(1/4-\vp)n\times(1/4-\vp)n$ of nonzero determinant. This
information will provide a non-equation.

Apply now the previous corollary to get a solution to our system
in $\overline\qz^{n^2|E|}$. Using this solution we construct a
map $\phi\colon E\to GL_n(\overline\qz)$ with the required properties.
\end{proof}

\begin{te}\label{lsoficityq}
Let $G$ be a linear sofic group. Then there exists an injective
morphism $\Theta\colon G\to\Pi_{k\to\omega}GL_{n_k}(\qz)$.
\end{te}
\begin{proof}
Fix a finite subset $E\subset G$ and $\vp>0$. By the previous
proposition, we obtain  a map $\phi\colon E\to GL_n(\overline\qz)$,
satisfying the algebraic definition of linear soficity. We 
replace $\overline\qz$ by $F$, the field generated by the $n^2|E|$
entries of the matrices in $\phi(E)$. Being a finitely generated
algebraic extension over $\qz$, the field $F$ is also a vector space over
$\qz$ of finite dimension. Then, we proceed as in
Observation \ref{vectorspaces}  and we get a required function $\psi\colon E\to
GL_n(\qz)$ having the same properties as in the algebraic definition.
\end{proof}

\section{Linear sofic groups and algebras}\label{sec:lin}

This section is devoted to proving Theorem \ref{mainresult} that a
group $G$ is linear sofic if and only if $\cz G$ is a linear sofic
algebra. While the ``if" part follows directly from Proposition~\ref{invertible in ultraproduct}, the ``only if" part is much more
involved.
\begin{nt}
If $\Theta\colon G\to\Pi_{k\to\omega} GL_{n_k}(F)/d_\omega$ is a group
morphism we denote by $\widetilde\Theta$ its extension to the
group algebra:
\[\widetilde\Theta\colon F(G)\to\Pi_{k\to\omega} M_{n_k}(F)/Ker \rho_\omega,\
\widetilde\Theta(\sum a_iu_{g_i}):=\sum a_i\Theta(g_i),\] where
$a_i\in F$, $g_i\in G$ and $u_{g_i}$ is the element in the group
algebra corresponding to $g_i$.
\end{nt}

\begin{ex}
If $\Theta$ is injective on $G$ it does not follow that
$\widetilde\Theta$ is injective on $F(G)$. As an easy example consider
$\Theta\colon \zz\to\Pi_{k\to\omega} GL_{k}(\rz)/d_\omega$,
$\Theta(i)=2^iId$, for $i\in\zz$. Then for $u_1-2u_0\in \rz(\zz)$
we have $\widetilde\Theta(u_1-2u_0)=2Id-2Id=0$.
\end{ex}

The proof relies on the direct sum and tensor product of elements
in ultraproduct of matrices. Here are variants of (4) and (5) of
Proposition \ref{prop of rk} extended to ultraproducts.

\begin{p}\label{prof of rk ultra}
Let $u=(u_k)_k\in\Pi_{k\to\omega} M_{n_k}(F)/Ker\rho_\omega$ and
$v=(v_k)_k\in \Pi_{k\to\omega} M_{m_k}(F)/Ker\rho_\omega$. Then:
\begin{align*}
&u\oplus v=(u_k\oplus v_k)_k\in \Pi_{k\to\omega}
M_{n_k+m_k}(F)/Ker\rho_\omega &; \rho_\omega(u\oplus
v)=&\frac{n_k\rho_\omega(u)+m_k\rho_\omega(v)}{n_k+m_k};\\
&u\otimes v=(u_k\otimes v_k)_k\in \Pi_{k\to\omega}
M_{n_km_k}(F)/Ker\rho_\omega&;\rho_\omega(u\otimes
v)=&\rho_\omega(u)\cdot\rho_\omega(v).
\end{align*}
\end{p}

\begin{te}\label{grouptoalgebra}
Let $\Theta\colon G\to\Pi_{k\to\omega} GL_{n_k}(F)/d_\omega$ be an
injective group morphism. Then there exists an injective algebra
morphism $\Psi\colon F(G)\to \Pi_{k\to\omega}
M_{m_k}(F)/Ker\rho_\omega$.
\end{te}
\begin{proof}
Let $\theta_k\colon G\to GL_{n_k}(F)$ be some functions such that
$\Theta=\Pi_{k\to\omega} \theta_k/d_\omega$. Then
$\Theta\otimes\Theta\colon G\to\Pi_{k\to\omega} GL_{n_k^2}(F)/d_\omega$,
defined by $\Theta\otimes\Theta(g)=\Pi
\theta_k(g)\otimes\theta_k(g)/d_\omega$ is a linear sofic
representation of $G$. For every $i\in\nz$ define a map
\[\theta_k^i\colon G\to\ GL_{n_k^i}(F),\ \
\theta_k^i(g):=\theta_k(g)\otimes\ldots\otimes\theta_k(g)\mbox{ (i
times tensor product)},\] and set
$\Theta^i=\Pi_{k\to\omega}\theta_k^i/d_\omega$. For $m\geqslant i$
define
\[\theta_k^{i,m}\colon G\to GL_{n_k^m}(F),\ \
\theta_k^{i,m}(g):=\theta_k^i(g)\otimes Id_{n_k^{m-i}}.\] The
meaning of this definition is to bring the first $m$
$\theta_k^i$'s into the same matrix dimension.

Now define $\phi_k\colon G\to GL_{n_k^k2^k}$ by:
\[\phi_k=(\theta_k^{1,k}\otimes Id_{2^{k-1}})\oplus(\theta_k^{2,k}\otimes
Id_{2^{k-2}})\oplus\ldots\oplus(\theta_k^{k,k}\otimes
Id_{2^0})\oplus(Id_{n_k^k}\otimes Id_{2^0})\] and set
$\Phi=\Pi_{k\to\omega}\phi_k/d_\omega$.

The reason for this definition is the relation:
\[\rho_\omega(\widetilde\Phi(f))=\sum_{i=1}^\infty\frac1{2^i}\rho_\omega(\widetilde\Theta^i(f)),\]
for any $f\in F(G)$. Before proving this equality let us state our
crucial claim.

\begin{claim}
$\widetilde\Phi\colon F(G)\to\Pi_{k\to\omega}
M_{n_k^k2^k}(F)/Ker\rho_\omega$ is injective.
\end{claim}

We now prove the stated relation:
\begin{align*}
\rho_\omega(\widetilde\Phi(f))=&\lim_{k\to\omega}\frac{rk(\tilde\phi_k(f))}{n_k^k2^k}=
\lim_{k\to\omega}\frac1{n_k^k2^k}\sum_{i=1}^krk\big(\tilde\theta_k^{i,k}\otimes
I_{2^{k-i}}(f)\big)\\ =&\lim_{k\to\omega}\sum_{i=1}^k
\frac1{2^in_k^k}rk\big(\tilde\theta_k^{i,k}(f)\big)=\lim_{k\to\omega}\sum_{i=1}^k
\frac1{2^in_k^i}rk\big(\tilde\theta_k^i(f)\big)\\
=&\sum_{i=1}^\infty
\frac1{2^i}\rho_\omega\big(\widetilde\Theta^i(f)\big)
\end{align*}

Assume now that $f\in F(G)$ such that
$\rho_\omega(\widetilde\Phi(f))=0$. Then
$\rho_\omega(\widetilde\Theta^i(f))=0$ for any $i$.

In order to present our injectivity argument in a transparent way
we shall assume that $f=a_1u_1+a_2u_2+a_3u_3$, where $a_i$ are
nonzero elements of $F$ and $u_i$ are invertible elements in the
group algebra corresponding to distinct elements in the group $G$.
We know that:
\begin{align}
&a_1\Theta(u_1)+a_2\Theta(u_2)+a_3\Theta(u_3)=0\\
&a_1\Theta(u_1)\otimes\Theta(u_1)+a_2\Theta(u_2)\otimes\Theta(u_2)+
a_3\Theta(u_3)\otimes\Theta(u_3)=0
\end{align}
\begin{align}
a_1\Theta(u_1)\otimes\Theta(u_1)\otimes\Theta(u_1)+
a_2\Theta(u_2)\otimes\Theta(u_2)\otimes\Theta(u_2)+
a_3\Theta(u_3)\otimes\Theta(u_2)\otimes\Theta(u_3)=0
\end{align}
Amplifying the first equation by $\Theta(u_3)$ and subtracting it
from the second we get:
\begin{align}
&a_1\Theta(u_1)\otimes(\Theta(u_1)-\Theta(u_3))+a_2\Theta(u_2)\otimes(\Theta(u_2)-\Theta(u_3))=0
\end{align}
Applying the same operation to equations $(2)$ and $(3)$, we get:
\begin{align}
&a_1\Theta(u_1)\otimes\Theta(u_1)\otimes(\Theta(u_1)-\Theta(u_3))+
a_2\Theta(u_2)\otimes\Theta(u_2)\otimes(\Theta(u_2)-\Theta(u_3))=0
\end{align}
Now we amplify equation $(4)$ with $\Theta(u_2)$ between the
already existing tensor product, and subtract it from equation
$(5)$ to get:
\begin{align}
&a_1\Theta(u_1)\otimes(\Theta(u_1)-\Theta(u_2))\otimes(\Theta(u_1)-\Theta(u_3))=0
\end{align}
As $a_1\neq 0$ and $\Theta(u_1)$ is invertible, we get that
$\Theta(u_1)=\Theta(u_2)$ or $\Theta(u_1)=\Theta(u_3)$. This
contradicts the injectivity of $\Theta$. This procedure applies to
any $f$ with an arbitrary large (finite) support.
\end{proof}

The key of the proof is the construction of the representation
$\widetilde\Phi$ out of a sequence of maps $\widetilde\Theta^i$ such that
$\widetilde\Phi(x)=0$ if and only if $\widetilde\Theta^i(x)=0$ for all $i$. This is a
construction that can be performed in general and we record it
here for a  later use.

\begin{p}\label{sum of representations}
Let $\{\Theta^i\}_i$, $\Theta^i\colon A\to\Pi_{k\to\omega}
M_{n_{i,k}}/Ker \rho_\omega$ be a sequence of morphisms of an
algebra $A$. Then there exists a morphism $\Phi$  of $A$ such that
$\rho_\omega(\Phi(x))=\sum_{i=1}^\infty\frac1{2^i}\rho_\omega(\Theta^i(x))$
for any $x\in A$. In particular, $\Phi(x)=0$ if and only if $\Theta^i(x)=0$
for all $i$. Moreover, if $\{\Theta^i\}_i$ are unital morphisms,
then $\Phi$ can be taken unital.
\end{p}

\begin{cor}
A group $G$ is linear sofic if and only if $\cz G$ is a linear
sofic algebra.
\end{cor}
\begin{proof}
The direct implication is the previous theorem for $F=\cz$. The
reverse implication immediately follows from Proposition
\ref{invertible in ultraproduct}.
\end{proof}

Our previous theorem also provides a new proof of the result of
Elek and Szabo~\cite{El-Sz1}.

\begin{cor}\label{sofkap}
Sofic groups satisfy Kaplansky's direct finiteness conjecture.
\end{cor}
\begin{proof}
Let $F$ be a field and $G$ be a sofic group. Same arguments as in
Proposition~\ref{sofic implies lsofic} show that there exists an injective
group morphism
$\Theta\colon G\to\Pi_{k\to\omega} GL_{n_k}(F)/d_\omega.$  The previous theorem provides an injective algebra
homomorphism $\Psi:F(G)\to\Pi_{k\to\omega} M_{m_k}(F)/
Ker\rho_\omega$. However, $\Pi_{k\to\omega} M_{m_k}(F)/Ker\rho_\omega$
is stably finite by Proposition~\ref{stably finite}. Thus, $F(G)$ is
actually stably finite in this case.
\end{proof}

\begin{q}\label{linsofkap} Do linear sofic groups satisfy Kaplansky's direct finiteness conjecture?
\end{q}

See also comments following Question \ref{q:oneF}  below.

\section{Linear sofic implies weakly sofic}\label{sec:wsofic}

Here we prove that a linear sofic group is weakly sofic. The proof
is an adaptation of the proof of Malcev's theorem\footnote{Malcev proves that every finitely generated subgroup of the linear group $GL_n(F)$ is residually finite.} presented in
 \cite[Theorem 1.4]{Pe-Kw}. Let us recall the
definition of weakly sofic group.

\begin{de}(c.f. \cite[Definition 4.1]{Gl-Ri})
A group $G$ is \emph{weakly sofic} if it can be embedded in a
metric ultraproduct of finite groups, each equipped with a
bi-invarant metric.
\end{de}

The original definition in \cite{Gl-Ri} is algebraic and uses a constant
length function (as discussed before Proposition \ref{algebraic definition}).
It is equivalent to its ultraproduct version above by standard amplification
argument \cite{Pe}. Indeed, the direct product of finite groups is obviously finite
and one can define a bi-invariant distance on the direct product as the sum of
the bi-invariant metrics on the factors.

\begin{te}
If $G$ is a linear sofic group then there exists $(F_k)_k$ a
sequence of finite fields and an injective group morphism
$\Phi\colon G\to\Pi_{k\to\omega} GL_{n_k}(F_k)/d_\omega$.
\end{te}
\begin{proof}
Let $\theta_k\colon G\to GL_{n_k}(\cz)$ be some functions such that
$\Theta=\Pi_{k\to\omega}\theta_k/d_\omega$ is an injective
homomorphism given by linear soficity of $G$. Let $G=\bigcup_kB_k$,
where $(B_k)_k$ is an increasing sequence of finite subsets of $G$
such that $B_k^{-1}=B_k$ and $e\in B_k$. Let $R_k\subset\cz$ be
the ring generated by all the entries of $\theta_k(s)$ with $s\in
B_k$. Because it is finitely generated, $R_k$ is a Jacobson ring.
We can view $\theta_k$ as a map from $B_k$ to $GL_{n_k}(R_k)$.
\begin{claim}
There exists $m_k\subset R_k$ a maximal ideal such that if we
reduce $\theta_k$ modulo $m_k$ to get the induced map
$\phi_k\colon B_k\to GL_{n_k}(R_k/m_k)$ we get:
\[rk(I-\theta_k(s))=rk(I-\phi_k(s))\ \ \forall s\in B_k.\]
\end{claim}
For $s\in B_k$ let $a_s=rk(I-\theta_k(s))$ and choose $A_s$ an
$a_s\times a_s$ submatrix of $\theta_k(s)$ such that $b_s=det
A_s\neq 0$. Let $c=\Pi_{s\in B_k}b_s$ and choose $m_k\subset R_k$
a maximal ideal such that $c\notin m_k$. Then $b_s\notin m_k$ for
any $s\in B_k$ so indeed $rk(I-\theta_k(s))=rk(I-\phi_k(s))$.

Since $m_k$ is a maximal ideal, $R_k/m_k$ is a field. It is a well-known
non-trivial fact that a finitely generated ring, that is also a field, is finite.
It follows that $R_k/m_k$ is finite.

Define $\Phi=\Pi_{k\to\omega}\phi_k/d_\omega$ and note that in
general if $s,t,st\in B_k$ then $rk(\phi_k(st)-\phi_k(s)\phi_k(t))
\leqslant rk(\theta_k(st)-\theta_k(s)\theta_k(t))$. This implies
that $\Phi$ is still a homomorphism and the claim shows that
$\Phi$ is injective.
\end{proof}

\begin{ob}
Every finite field $F$ is a finite dimensional vector space over
$\zz/p\zz$, where $p$ is the characteristic of $F$. Therefore, as in
Observation \ref{vectorspaces}, if we have an embedding
$\phi\colon G\to\Pi_{k\to\omega} GL_{n_k}(F_k)/d_\omega$ with $F_k$
finite fields, then we can construct $\psi\colon G\to\Pi_{k\to\omega}
GL_{m_k}(\zz/p_k\zz)/d_\omega$, where $(p_k)_k$ is a sequence of
prime numbers.
\end{ob}

\begin{q}
Are all linear sofic groups indeed sofic?
\end{q}

For this question the tensor product is not a useful tool. Suppose that we have
a map $\theta\colon E\to GL_n(\cz)$ from a finite subset $E$ of a linear sofic group $G$.
We want to construct a new map from $E$ into $S_n$. As permutation matrices are diagonalizable,
we can first try to construct a map using only diagonalizable matrices.

If $A\in M_n(\cz)$ let $A=UTU^{-1}$ be the canonical Jordan decomposition of $A$. Let $Diag(T)$ be the
diagonal matrix obtained by taking only the entries on the diagonal of $T$. Then $\rho(A-UDiag(T)U^{-1})=1-J(A)$,
where $J(A)$ is the number of Jordan blocks in $A$ divided by $n$ as defined in Section \ref{sec:rampl}.
In the rank metric, this is the lower bound for $\rho(A-D)$, where $D$ is any diagonalizable matrix. This follows from Theorem 2 of \cite{Gl-Ri-ar}
or it can be checked directly. In Section \ref{sec:rampl}, we proved that $J(A\otimes A)\leqslant J(A).$ Therefore, taking the
tensor product will only increase the rank distance from $A$ to a diagonalizable matrix, not reduce it.

\begin{q}\label{q:oneF}
Let $G$ be a linear sofic group and  $F$ a
finite field. Does there exist an injective group morphism
$\Phi:G\to\Pi_{k\to\omega} GL_{n_k}(F)/d_\omega$?
\end{q}

Sofic groups have this property and this is the only property that we used
in our proof of Kaplansky's direct finiteness conjecture, Corollary \ref{sofkap}.
A positive answer to this question will immediately imply that linear sofic groups do satisfy
Kaplansky's direct finiteness conjecture. That would give a positive answer to Question \ref{linsofkap}.

\section{Permanence properties}\label{sec:perm}

Here we shall prove various permanence properties for linear sofic
groups and algebras. Due to Theorem \ref{mainresult} many
permanence properties for linear sofic algebras can be transported
to linear sofic groups.

\begin{te}
Subalgebras, direct product, inverse limits of linear sofic
algebras are linear sofic. Same permanence properties hold also
for linear sofic groups.
\end{te}
\begin{proof}
It is not hard to see that a subalgebra of a linear
sofic algebra is linear sofic. Let $(A_i)_i$ be a sequence of
linear sofic algebras and let $A=\Pi_i A_i$ be its direct product.
Denote by $P_j\colon\Pi_iA_i\to A_j$ the projection to the $j$-th
component. Let $\Theta_i\colon A_i\to\Pi_{k\to\omega} M_{n_{i,k}}/Ker
\rho_\omega$ be a linear sofic representation of $A_i$. Then
$(\Theta_i\circ P_i)_i$ is a sequence of morphisms of the algebra
$A$. Using Proposition \ref{sum of representations} we construct
$\Psi\colon A\to\Pi_{k\to\omega} M_{n_k}/Ker \rho_\omega$ such that
$Ker\Psi=\bigcap_i Ker(\Theta_i\circ P_i)=0$. It follows that $A$ is
linear sofic. An inverse limit is a specific subalgebra of the
direct product.

The second part of the theorem follows immediately from Theorem \ref{mainresult}
and properties of group algebras for these constructions.
\end{proof}

\begin{te}
Direct limit of linear sofic groups is again sofic.
\end{te}
\begin{proof}
For the proof we use Proposition \ref{16 algebraic
definition}. Note that Proposition \ref{algebraic
definition} is not sufficient for this result.

Let $\{G_i\}_i$ be a family of linear sofic groups together with
morphisms required to construct the group $G$, the direct
limit of this family. Let $\psi_i\colon G_i\to G$ be the morphisms provided by
the definition of the direct limit.

Let $E$ be a finite subset of $G$ and $\vp>0$. There exists
$i_0\in\nz$ and $E_0\subset G_{i_0}$ such that $\psi_{i_0}\colon E_0\to
E$ is a bijection. Apply Proposition \ref{16 algebraic definition}
for $G_{i_0}$, $E_0$ and $\vp$ to get $\phi\colon E_0\to GL_n(\cz)$.
Then $\phi\circ\psi_{i_0}^{-1}\colon E\to GL_n(\cz)$ is the required
function for $G$, $E$ and~$\vp$.
\end{proof}

Elek and Szabo proved that amenable extensions of sofic groups is
again sofic \cite{El-Sz}. The same result is true for linear sofic
groups and our proof is a careful adaptation of the proof of the sofic case presented in
\cite{Oz}.

\begin{te}
Let $G$ be a countable group and $H$ a normal subgroup of $G$. If
$H$ is linear sofic and $G/H$ is amenable then $G$ is linear
sofic.
\end{te}
\begin{proof}
Let $\sigma\colon G/H\to G$ be a lift and define $\alpha\colon G\times G/H\to
H$ by $\alpha(g,\gamma)=\sigma(g\gamma^{-1})g \sigma(\gamma)$.
Then $\alpha$ satisfies the cocycle identity,
$\alpha(g_1g_2,\gamma)=\alpha(g_1,g_2\gamma)\alpha(g_2,\gamma)$.

Let $E\subset G$ be a finite subset and $\vp>0$. Let $F\subset G/H$
be such that $|gF\cap F|>(1-\vp)|F|$ for all $g\in E$. Then
$\alpha(E,F)\subset H$ is finite. Use Proposition \ref{16
algebraic definition} and the linear soficity of $H$ to get $\phi\colon \alpha(E,F)\to GL_n(\cz)$.
Construct $\psi\colon E\to M_{n\cdot|F|}(\cz)$ by:
\[\psi(g)=\sum_{\gamma\in F\cap
g^{-1}F}\phi(\alpha(g,\gamma))\otimes e_{g\gamma,\gamma}.\] Here
$e_{g\gamma,\gamma}\in M_{|F|}(\cz)$ is a unit matrix, that is a matrix
having only one entry of $1$ on the position $(g\gamma,\gamma)$.
It is easy to compute $\rho(\psi(g))=|F\cap g^{-1}F|/|F|>1-\vp$,
so $\psi(g)$ is almost an element of $GL_{n|F|}$.

We want to show that $\psi(g_1)\psi(g_2)$ is close to
$\psi(g_1g_2)$. By construction:
\[\psi(g_1)\psi(g_2)=\sum_{\gamma_1\in F\cap g_1^{-1}F}\sum_{\gamma_2\in
F\cap g_2^{-1}F}\phi(\alpha(g_1,\gamma_1))
\phi(\alpha(g_2,\gamma_2))\otimes e_{g_1\gamma_1,\gamma_1}
e_{g_2\gamma_2,\gamma_2}.\] Inside the sum we must have
$\gamma_1=g_2\gamma_2$ in order to get a non trivial term.
\[\psi(g_1)\psi(g_2)=\sum_{\gamma_2\in F\cap
g_2^{-1}F\cap(g_1g_2)^{-1}F}\phi(\alpha(g_1,g_2\gamma_2))
\phi(\alpha(g_2,\gamma_2))\otimes e_{g_1g_2\gamma_2,\gamma_2}.\]
Also $\psi(g_1g_2)=\sum_{\gamma\in
F\cap(g_1g_2)^{-1}F}\phi(\alpha(g_1,g_2\gamma)\alpha(g_2,\gamma))\otimes
e_{g_1g_2\gamma,\gamma}$. Comparing the two equations we get:
\[\rho(\psi(g_1)\psi(g_2)-\psi(g_1g_2))\leqslant\frac1{n|F|}(|F|n\vp+\vp|F|)
<2\vp.\] We only need to show that $\rho(Id-\psi(g))$ is larger
than a constant. If $g\in H$ then:
\[\rho(Id-\psi(g))=\frac1{|F|}\rho(Id-\sum_{\gamma\in
F}\phi(\alpha(g,\gamma))\otimes
e_{\gamma,\gamma})=\frac1{|F|}\sum_{\gamma\in
F}\rho(Id-\phi(\alpha(g,\gamma))\geqslant\frac14-\vp.\] Consider
now the case $g\notin H$ such that $g\gamma\neq\gamma$ for any
$\gamma\in F$. Let $x=(x_\gamma)_\gamma\in\cz^{n\cdot|F|}$ be a
vector in $Ker(Id-\psi(g))$. An easy computation will provide the
equation $\phi(\alpha(g,\gamma))(x_\gamma)=x_{g\gamma}$ for
$\gamma\in F\cap g^{-1}F$. This means that if we fix $x_\gamma$
then $x_{g\gamma}$ is completely determined. It follows that $\dim
Ker(Id-\psi(g))$ can not be greater than $1/2$ minus some $\vp$
due to the restriction $\gamma\in F\cap g^{-1}F$. So
$\rho(Id-\psi(g))=1-\dim Ker(Id-\psi(g))\geqslant 1/2-\vp$.
\end{proof}

\section{The number of universal linear sofic groups}\label{sec:universal}

A \emph{universal linear sofic} group is a metric ultraproduct of
$(GL_{n_k}(F))_k$ as defined in Definition \ref{universallinearsofic}. In
\cite{Lu} Lupini proved that, under the failure of the Continuum
Hypothesis ($CH$), there are $2^{\aleph_c}$ metric ultraproducts
of matrix algebras endowed with the metric induced by the rank
(Definition \ref{rankultraproduct}), up to algebraic isomorphism.
This result is based on methods of continuous logic developed in \cite{Fa-Sh}. Here we
extend Lupini's arguments to show that, assuming $\neg CH$, there
are $2^{\aleph_c}$ universal linear sofic groups. Such results are
not known for general weakly sofic groups when the approximating family of
finite groups endowed with bi-invariant metrics is given.

 Recall that by definition $\aleph_c:=2^{\aleph_0}$, where $\aleph_0$ is the
cardinality of $\nz$. If $a,b$ are elements of a group then
$[a,b]$ is defined as $aba^{-1}b^{-1}$.

In this section, $(n_k)_k\subset\nz$ is a fixed strictly
increasing sequence. We obtain non-isomorphic universal
linear sofic groups by using different ultrafilters.

\begin{p}(\cite[Corollary 2]{Lu})\label{manyultraproducts}
Let $(G_n)_n$ be a sequence of groups, each equipped with a
bi-invarant metric, with uniformly bounded diameter. Suppose that
for some constant $\gamma>0$ and every $l\in\nz$, for all but
finitely many $n\in\nz$, $G_n$ contains sequences
$(g_{n,i})_{i=1}^l$ and $(h_{n,i})_{i=1}^l$ such that, for every
$1\leqslant i<j\leqslant l$, $g_{n,i}$ and $h_{n,j}$ commute,
while if $1\leqslant j\leqslant i\leqslant l$,
$d([g_{n,i},h_{n_,j}];e_{G_n})\geqslant\gamma$. Then under the
failure of $CH$, there are $2^{\aleph_c}$ many pairwise non
isometrically isomorphic metric ultraproducts of the sequence
$(G_{n_k})_{k\in\nz}$.
\end{p}

Lupini used this proposition for $(S_n,d_{Hamm})_n$ to show that
there are $2^{\aleph_c}$ many universal sofic groups. We shall use
the same permutation that he constructed, regarded now as elements
in $(GL_n(\cz),d_{rk})_n$ to show that the hypothesis of the
proposition still holds for these groups.

\begin{p}
The hypothesis of Proposition \ref{manyultraproducts} holds for
$GL_n(\cz)$ endowed with the bi-invariant metric $d_{rk}$. The
constant $\gamma$ can be chosen $2/9$.
\end{p}
\begin{proof}
As $(GL_n(\cz),d_{rk})_n$ are bi-invariant metric groups with
uniformly bounded diameter we just need to construct elements
$g_{n,i}$ and $h_{n,i}$.

First assume that $n=3^l$ for some $l\in\nz$. Let $(12)$ and
$(23)$ be two transpositions in $S_3$ and denote by $A_{(12)}$ and
$A_{(23)}$ the corresponding permutation matrices in $GL_3(\cz)$.
For $1\leqslant i\leqslant l$ define $g_{n,i}=A_{(12)}\otimes
A_{(12)}\otimes\ldots\otimes A_{(12)}\otimes Id_{3^{l-i}}$
($A_{(12)}$ is used $i$ times) and $h_{n,i}=Id_{3^{i-1}}\otimes
A_{(23)}\otimes Id_{3^{l-i}}$. It is easy to check that for $i<j$
$g_{n,i}$ and $h_{n,j}$ commutes, while for $i\geqslant j$
$[g_{n,i},h_{n,j}]=Id_{3^{j-1}}\otimes A_{(123)}\otimes
Id_{3^{l-j}}$. This means that $[g_{n,i},h_{n,j}]$ is composed of
$3^{l-1}$ cycles of length $3$, so
$d_{rk}([g_{n,i},h_{n,j}],Id_{3^l})=1-3^{l-1}/3^l=2/3$.

Let now $n\in\nz$ be an arbitrary number and $l\in\nz$ such that
$3^l\leqslant n<3^{l+1}$. Define $g_{n,i}=g_{3^l,i}\oplus
Id_{n-3^l}$ and $h_{n,i}=h_{3^l,i}\oplus Id_{n-3^l}$. Again for
$i<j$ $g_{n,i}$ and $h_{n,j}$ commutes, while for $i\geqslant j$
$[g_{n,i},h_{n,j}]=Id_{3^{j-1}}\otimes A_{(123)}\otimes
Id_{3^{l-j}}\oplus Id_{n-3^l}$. Then
$d_{rk}([g_{n,i},h_{n,j}],Id_{3^l})=1-(3^{l-1}+n-3^l)/n=(3^l-3^{l-1})/n\geqslant
2/9$. Thus, the constant $\gamma$ can be set $2/9$.
\end{proof}

\section{Almost finite dimensional representations}\label{sec:almostfd}

In this section, we work only with unital algebras.  The following propery of
algebras was introduced by Gabor Elek.

\begin{de}(\cite[Definition 1.1]{El})
A unital $F$-algebra $A$ has \emph{almost finite dimensional
representations} if for any finite dimensional subspace $1\in
L\subset A$ and $\vp>0$, there exists a finite dimensional vector
space $V$ together with a subspace $V_\vp\subset V$ such that
\begin{enumerate}
\item there exists a linear (not necessarily injective) map
$\psi_{L,\vp}\colon L\to End_F(V)$ such that $\psi_{L,\vp}(1)=Id$ and
$\psi_{L,\vp}(a)\psi_{L,\vp}(b)(v)=\psi_{L,\vp}(ab)(v)$ for
$a,b,ab\in L$ and $v\in V_\vp$. \item $\dim_FV-\dim_FV_\vp<\vp\cdot
\dim_FV$.
\end{enumerate}
Such a map is called an \emph{ $\vp$-almost representation} of $L$.
\end{de}

\begin{p}
A unital algebra $A$ has almost finite dimensional representations
if and only if there exists a unital morphism (not necessarily injective)
$\Theta\colon A\to\Pi M_{n_k}(F)/Ker\rho_\omega$.
\end{p}
\begin{proof}
Let $(L_k)_k$ be an increasing sequence of finite dimensional
subspaces of $A$ such that $A=\bigcup_kL_k$ and $1\in L_k$. Let
$(\vp_k)_k$ be a decreasing sequence of strictly positive reals
such that $\lim_k\vp_k=0$. For every $k$, let
$\psi_{L_k,\vp_k}\colon L_k\to End_F(V_k)$ be the map from the
previous definition. Define $n_k=\dim_FV_k$. Then $\Theta\colon A\to\Pi
M_{n_k}(F)/Ker\rho_\omega$ defined by
$\Theta=\Pi\psi_{L_k,\vp_k}/Ker \rho_\omega$ is a unital morphism.
The reverse implication follows from the definition of
ultraproduct.
\end{proof}

If we compare this definition to the definition of linear sofic
algebras we see that having almost finite dimensional
representations is the first step towards linear soficity. However,
this is not sufficient. We introduce an object that measures how far
from being linear sofic is an algebra with almost finite
dimensional representations. This is inspired by the definition of the rank radical by Elek.

\subsection{The rank radical}

\begin{de}(\cite[Definition 4.1]{El})
The \emph{rank radical} $RR(A)$ of an algebra is defined as
follows: if $A$ does not have almost finite dimensional
representations then $RR(A)=A$. Otherwise, let $p\in RR(A)$ if
there exists a finite dimensional subspace $L$ with
$\{1,p\}\subset L\subset A$ such that for any $\delta>0$ there
exists $n_\delta>0$ with the following property: if
$0<\vp<n_\delta$ and $\psi_{L,\vp}\colon L\to End(V)$ is an $\vp$-almost
representation then $\dim Ran(\psi_{L,\vp}(p))<\delta\cdot \dim V$.
\end{de}

We restate this property in ultraproduct language. We use the following definition.

\begin{de}
If $1\in L\subset A$ is a linear subspace of an algebra, then a
\emph{partial morphism} of $L$ is a linear function $\Phi\colon L\to\Pi
M_{n_k}(F)/Ker\rho_\omega$ such that $\Phi(1)=1$ and
$\Phi(x)\Phi(y)=\Phi(xy)$ whenever $x,y,xy\in L$.
\end{de}

\begin{p}\label{description of RR}
For any element $p$ of an algebra $p\in RR(A)$ if and only if there exists a
finite dimensional subspace $L$ with $\{1,p\}\subset L\subset A$
such that for any partial morphism $\Phi\colon L\to\Pi
M_{n_k}(F)/Ker\rho_\omega$ we have $\Phi(p)=0$.
\end{p}
\begin{proof}
We first assume that $A$ has almost finite dimensional
representations. Note that this is equivalent to
$RR(A)\varsubsetneq A$.

Let $p\in RR(A)$. Let $L$ be the finite dimensional subspace from
the definition of the rank radical. Fix $\delta>0$ and use again
the definition to get a $n_\delta>0$. Let $\Phi\colon L\to\Pi
M_{n_k}(F)/Ker\rho_\omega$ be a partial morphism,
$\Phi=\Pi\phi_k/Ker\rho_\omega$. Let $0<\vp<n_\delta$. Then there
exists $H\in\omega$ such that $\phi_k\colon L\to M_{n_k}(F)$ is an
$\vp$-almost representation of $L$ for any $k\in H$. Then $\dim
Ran(\phi_k(p))<\delta n_k$, or with our notation
$\rho(\phi_k(p))<\delta$ for any $k\in H$. This implies that
$\rho_\omega(\Phi(p))<\delta$. As $\delta$ was arbitrary it
follows that $\rho_\omega(\Phi(p))=0$ so $\Phi(p)=0$.

Suppose now $p\notin RR(A)$. Let $L$ be an arbitrary finite
dimensional subspace. Then there exists $\delta_L>0$ such that for
any $\vp>0$ there exists an $\vp$-almost representation
$\psi_{L,\vp}\colon L\to End(V)$  with $\dim
Ran(\psi_{L,\vp}(p))\geqslant\delta_L\cdot \dim V$. This is
equivalent to $\rho(\psi_{L,\vp}(p))\geqslant\delta_L$.

Let $(\vp_k)_k$ be a decreasing sequence converging to $0$ and
$\psi_{L,\vp_k}\colon L\to End(V_k)$ be $\vp_k$-almost representations
with $\dim \rho(\psi_{L,\vp_k}(p))\geqslant\delta_L$. Define
$\Psi=\Pi\psi_{L,\vp_k}/Ker\rho_\omega$. Because $\vp_k\to 0$,
$\Psi$ is a partial morphism of $L$. Also
$\rho_\omega(\Psi(p))\geqslant\delta_L$ so $\Psi(p)\neq 0$.

Consider now the case $RR(A)=A$. Let $p\in A$. We shall prove that
there exists a finite dimensional subspace $L$ with
$\{1,p\}\subset L\subset A$ such that there is no partial morphism
$\Phi:L\to\Pi M_{n_k}(F)/Ker\rho_\omega$.

Let $A=\bigcup_iL_i$ where $(L_i)_i$ is an increasing sequence of
finite dimensional subspaces of $A$ such that $\{1,p\}\in L_i$.
Assume that for each $i$ there exists a partial morphism
$\Phi_i:L_i\to\Pi M_{n_{i,k}}(F)/Ker\rho_\omega$. Consider also
$(\vp_i)_i$ a sequence of strictly positive real numbers
converging to $0$.

The existence of $\Phi_i$ implies the existence of $\psi_i:L_i\to
M_{n_i,k_i}$ an $\vp_i$-almost representation of $L_i$. Define
$\Theta=\Pi\psi_i/Ker\rho_\omega$. Because $A=\bigcup_iL_i$ and
$\vp_i\to 0$, $\Theta$ is a unital morphism. This is in
contradiction with the fact that $A$ does not have almost finite
dimensional representations.
\end{proof}

Elek proved that $RR(A)$ is an ideal. We can deduce this from our
description.

\begin{cor}
The set $RR(A)$ is an ideal.
\end{cor}
\begin{proof}
Let $a\in A$ and $p\in RR(A)$. Let $\{1,p\}\subset L_p\subset A$
be a finite dimensional subspace such that $\Phi(p)=0$ for any
partial morphism $\Phi\colon L_p\to\Pi M_{n_k}(F)/Ker\rho_\omega$.

Define $L_{ap}=Sp\{L_p\cup\{a,ap\}\}$, defined by thaking the linear span. Let $\Psi\colon L_{ap}\to\Pi
M_{n_k}(F)/Ker\rho_\omega$ be a partial morphism. Then
$\Psi(ap)=\Psi(a)\Psi(p)$. But $L_p\subset L_{ap}$ and $p\in
RR(A)$ implies $\Psi(p)=0$. So $\Psi(ap)=0$. The same proof works
for $pa$.
\end{proof}

\begin{te}
The rank radical of $A/RR(A)$ is $0$.
\end{te}
\begin{proof}
We denote by $f\colon A\to A/RR(A)$ the canonical projection. Let
$v\in A/RR(A)$, $v\neq 0$. Let $\{1,v\}\subset L$ be a finite
dimensional subspace of $A/RR(A)$. Choose $N\subset A$ a finite
dimensional subspace such that $f(N)=L$ and $1\in N$. There exists
$u\in N$ such that $v=f(u)$ and $u\notin RR(A)$. Define
$N_0=Sp\{N\cup N^2\}\cap RR(A)$. Then $N_0$ is finite dimensional
and choose $\{z_1,\ldots,z_r\}$ a base for $N_0$. For any
$1\leqslant i\leqslant r$ there exists $L_i\subset A$ a finite
dimensional subspace such that for any partial morphism
$\Phi\colon L_i\to\Pi M_{n_k}/Ker\rho_\omega$ $\Phi(z_i)=0$.

Define $N_1=Sp\{N\cup N^2\cup\bigcup_i L_i\}$. Because $u\notin
RR(A)$ there exists a partial morphism $\Phi\colon N_1\to\Pi
M_{n_k}/Ker\rho_\omega$ such that $\Phi(u)\neq 0$. As $L_i\subset
N_1$ we get $\Phi(z_i)=0$ for any $i$. This implies that
$\Phi(N_0)=0$ so we can factor $\Phi$ to get a linear function
$\Psi\colon L\cup L^2\to\Pi M_{n_k}/Ker\rho_\omega$. If $a,b\in L$ then
$\Psi(a)\Psi(b)=\Psi(ab)$ so $\Psi$ restricted to $L$ is a partial
morphism. Also $\Psi(v)=\Phi(u)\neq 0$. It follows that $v\notin
RR(A/RR(A))$.
\end{proof}

\begin{p}(\cite[Proposition 4.3]{El})
Let $A$ be an algebra such that $RR(A)=0$. Then $A$ is stably
finite.
\end{p}
\begin{proof}
We simplify the original proof by the use of ultrafilters. First we prove that if $RR(A)=0$ then
$RR(M_m(A))=0$. Recall that $M_m(A)\simeq M_m(F)\otimes A$. Let
$v\in M_m(A)$, $v\neq 0$ and let $u\in A$ be a nonzero entry of
$v$. Consider $\{1,v\}\subset L\subset M_m(A)$ a finite
dimensional subspace. Then there exists $\{1,u\}\subset L_1\subset
A$ a finite dimensional subspace such that $L\subset M_m(F)\otimes
L_1$. As $u\notin RR(A)$ there exists $\Phi_1\colon L_1\to\Pi
M_{n_k}(F)/Ker \rho_\omega$ such that $\Phi(u)\neq 0$. Define
$\Phi\colon M_m(F)\otimes L_1\to\Pi M_{m\cdot n_k}(F)/Ker \rho_\omega$
by $\Phi(a\otimes p)=a\otimes\Phi_1(p)$. Then $\Phi(v)\neq 0$.

Consider now $x,y\in A$ such that $xy=1$. Assume that $yx\neq 1$
so $xy-yx\neq 0$. Let $L=Sp\{1,x,y,yx\}$. As $xy-yx\notin RR(A)$
there exists a partial morphism $\Phi:L\to\Pi M_{n_k}(F)/Ker
\rho_\omega$ such that $\Phi(xy-yx)\neq 0$. Now
$1=\Phi(xy)=\Phi(x)\Phi(y)$ and by Proposition \ref{stably
finite}, $\Pi M_{n_k}(F)/Ker \rho_\omega$ is directly finite. Thus,
$\Phi(y)\Phi(x)=1$. It follows that $\Phi(xy-yx)=1-1=0$
contradiction.
\end{proof}

\subsection{The sofic radical}

It is easy to see that if an algebra $A$ is linear sofic then
$RR(A)=0$. However, this condition is not sufficient. We modify the
definition of the rank radical to get a larger ideal
that will describe linear soficity. This can be done also for
groups and we first discuss this case as a warm up.

\begin{de}
The \emph{linear sofic radical} $LSR(G)$ of a group $G$ is defined as follows: $h\in
LSR(G)$ whenever for all group morphisms $\Theta\colon G\to\Pi
GL_{n_k}(\cz)/d_\omega$ we have $\Theta(h)=1$.
\end{de}

\begin{p}\label{linear sofic radical}
The linear sofic radical $LSR(G)$ is a normal subgroup of $G$. The
group $G/LSR(G)$ is linear sofic.
\end{p}
\begin{proof}
It it easy to see from the definition that:
\[LSR(G)=\bigcap\{Ker\Theta\mid \Theta\colon G\to\Pi
GL_{n_k}(\cz)/d_\omega\mbox{ group morphism}\},\] so indeed
$LSR(G)$ is a normal subgroup of $G$.

For the second part of the proposition, for each $g\in G$,
$g\notin LSR(G)$ consider a morphism $\Theta_g\colon G\to\Pi
GL_{n_k}(\cz)/d_\omega$ such that $\Theta_g(g)\neq 1$. Then
$\{\widetilde\Theta_g\}_{g\in G\setminus LRS(G)}$ is a sequence of
unital morphisms of the group algebra and we apply Proposition
\ref{sum of representations} to get a morphism $\Theta\colon G\to\Pi
GL_{m_k}(\cz)/d_\omega$ such that $Ker\Theta=\bigcap_{g\in G\setminus
LRS(G)} Ker\Theta_g=LRS(G)$.
\end{proof}

This construction of linear sofic radical can be performed for
several other metric approximation properties for groups, like soficity,
weak soficity or hyperlinearity. We now introduce the sofic radical for algebras.

\begin{de}
The \emph{sofic radical} $SR(A)$ of an algebra is defined as
follows: if $A$ does not have almost finite dimensional
representations then $SR(A)=A$. Otherwise, let $p\in SR(A)$ if for
any $\delta>0$ there exists a finite dimensional subspace $L$ with
$\{1,p\}\subset L\subset A$ and there exists $n_\delta>0$ with the
following property: if $0<\vp<n_\delta$ and $\psi_{L,\vp}\colon L\to
End(V)$ is an $\vp$-almost representation then $\dim
Ran(\psi_{L,\vp}(p))<\delta\cdot \dim V$.
\end{de}

We now provide a characterization of the sofic radical in terms of
morphisms into ultraproducts.

\begin{p}\label{description of SR}
For any element $p$ of an algebra $p\in SR(A)$ if and only if for any unital
morphism $\Theta\colon A\to\Pi M_{n_k}(F)/Ker\rho_\omega$ we have
$\Theta(p)=0$.
\end{p}
\begin{proof}
Let $p\in SR(A)$. Fix $\delta>0$ and let $L$ and $n_\delta>0$ as
in the definition of the sofic radical. Let $\Theta:A\to\Pi
M_{n_k}(F)/Ker\rho_\omega$ be a unital morphism,
$\Theta=\Pi\theta_k/Ker\rho_\omega$. Let $0<\vp<n_\delta$. Then
there exists $H\in\omega$ such that $\theta_k\colon L\to M_{n_k}(F)$ is
an $\vp$-almost representation of $L$ for any $k\in H$. Then $\dim
Ran(\theta_k(p))<\delta n_k$, or with our notation
$\rho(\theta_k(p))<\delta$ for any $k\in H$. This implies that
$\rho_\omega(\Theta(p))<\delta$. As $\delta$ was arbitrary it
follows that $\rho_\omega(\Theta(p))=0$ so $\Theta(p)=0$.

Suppose now $p\notin SR(A)$. Then there exists $\delta>0$ such
that for any finite dimensional subspace $L$ with $\{1,p\}\subset
L\subset A$ and any $n>0$ there exists $0<\vp<n$ and
$\psi_{L,\vp}\colon L\to End(V)$ an $\vp$-almost representation with
$\dim Ran(\psi_{L,\vp}(p))\geqslant\delta\cdot \dim V$. This is
equivalent to $\rho(\psi_{L,\vp}(p))\geqslant\delta$.

Let $A=\bigcup_kL_k$ where $(L_k)_k$ is an increasing sequence of
finite dimensional subspaces of $A$ such that $\{1,p\}\in L_k$. We
can find a decreasing sequence $(\vp_k)_k$ converging to $0$
and $\vp_k$-almost
representations $\psi_{L_k,\vp_k}\colon L_k\to End(V_k)$  with $\dim
\rho(\psi_{L_k,\vp_k}(p))\geqslant\delta$.

Define $\Theta=\Pi\psi_{L_k,\vp_k}/Ker\rho_\omega$. Since
$A=\cup_kL_k$ and $\vp_k\to 0$, $\Theta$ is a unital morphism.
Also $\rho_\omega(\Theta(p))\geqslant\delta$ so $\Theta(p)\neq 0$.
\end{proof}

\begin{cor}
The sofic radical is an ideal. Moreover, $SR(A/SR(A))=0$.
\end{cor}
\begin{proof}
By the previous proposition,
\[SR(A)=\bigcap\{Ker\Theta\mid\Theta:A\to\Pi
M_{n_k}(F)/Ker\rho_\omega,\mbox{ unital morphism}\}\] Let now
$q\in A/SR(A)$ and $p\in A$ such that $q=\hat{p}$. Suppose that
$q\neq 0$ so that $p\notin SR(A)$. Then there exists a unital
morphism $\Theta\colon A\to\Pi M_{n_k}(F)/Ker\rho_\omega$ such that
$\Theta(p)\neq 0$. Also $SR(A)\subset Ker\Theta$. So we can factor
$\Theta$ by $SR(A)$ to get a unital morphism
$\widetilde{\Theta}\colon A/SR(A)\to\Pi M_{n_k}(F)/Ker\rho_\omega$ with
$\widetilde{\Theta}(q)\neq 0$. It follows that $q\notin
SR(A/SR(A))$.
\end{proof}

\begin{te}
An algebra $A$ is linear sofic if and only if $SR(A)=0$.
\end{te}
\begin{proof}
If $A$ is linear sofic then there exists an injective morphism
$\Theta\colon A\to\Pi M_{n_k}(F)/Ker\rho_\omega$. So $Ker\Theta=0$ and
then $SR(A)=0$.

Let $A$ be an algebra such that $SR(A)=0$. It follows that for any
$p\neq 0$ there exists a unital morphism $\Psi_p\colon A\to\Pi
M_{n_{k,p}}(F)/Ker\rho_\omega$ such that $\Psi_p(p)\neq 0$.

Let $\{x_i\}_{i\in\nz}$ be a basis for $A$ as a vector space over
$F$. For $s\in\nz$, we shall inductively construct $\Phi_s\colon A\to\Pi
M_{n_{k,s}}(F)/Ker\rho_\omega$ such that $Ker\Phi_s\cap
Sp\{x_1,\ldots,x_s\}=0$.

Define $\Phi_1=\Psi_{x_1}$. Assume now by induction that we have
$\Phi_{s-1}$ unital morphism such that $Ker\Phi_{s-1}\cap
Sp\{x_1,\ldots,x_{s-1}\}=0$. Then $\dim(Ker\Phi_{s-1}\cap
Sp\{x_1,\ldots,x_{s}\})\leqslant 1$. If this space is trivial
define $\Phi_s=\Phi_{s-1}$. Otherwise, let $y_s\in
Ker\Phi_{s-1}\cap Sp\{x_1,\ldots,x_{s}\}$, $y_s\neq 0$ and define
$\Phi_s=\Phi_{s-1}\oplus\Psi_{y_s}$ (see Proposition \ref{prof of
rk ultra}). If $z\in Ker\Phi_s\cap Sp\{x_1,\ldots,x_s\}$ then
$z\in Ker\Phi_{s-1}\cap Sp\{x_1,\ldots,x_{s}\}$. It follows that
$z\in Sp\{y_s\}$. But also $z\in Ker\Psi_{y_s}$ so $z=0$.

Using arguments similar to the proof of Theorem \ref{grouptoalgebra}
(see also Proposition \ref{sum of representations}), we shall
construct a unital morphism $\Theta$ such that
$Ker\Theta=\cap_sKer\Phi_s$.

First we bring $\Phi_s$ into the same sequence of matrix
dimensions. Define $n_k=n_{k,1}n_{k,2}\ldots n_{k,k}$. Replace
$\Phi_s$ by an amplification to get $\Phi_s:A\to\Pi
M_{n_k}(F)/Ker\rho_\omega$.

Let now $\phi_{s,k}:A\to M_{n_k}$ be such that
$\Phi_s=\Pi\phi_{s,k}/Ker\rho_\omega$. Define $\theta_k:A\to
M_{2^kn_k}$ by:
\[\theta_k=(\phi_{1,k}\otimes Id_{2^{k-1}})\oplus(\phi_{2,k}\otimes
Id_{2^{k-2}})\oplus\ldots\oplus(\phi_{k,k}\otimes Id_{2^0})\oplus
Id_{n_k},\] and $\Theta=\Pi\theta_k/Ker\rho_\omega$. As in the
proof of Theorem \ref{grouptoalgebra}
$\rho_\omega(\Theta(p))=\sum_{s=1}^\infty\frac1{2^s}\rho_\omega(\Phi_s(p))$.
It follows that indeed $Ker\Theta=\bigcap_sKer\Phi_s$. As
$A=\bigcup_sSp\{x_1,\ldots,x_s\}$ and $Ker\Phi_s\bigcap
Sp\{x_1,\ldots,x_s\}=0$ we get $Ker\Theta=0$.
\end{proof}

\begin{cor}
A simple unital algebra with almost finite dimensional
representations is linear sofic.
\end{cor}
\begin{proof}
If $A$ has almost finite dimensional representations then $1\notin
SR(A)$. As $SR(A)$ is an ideal and $A$ is simple we get $SR(A)=0$.
By the previous theorem $A$ is linear sofic.
\end{proof}

\begin{p}
An amenable algebra without zero divisors is linear sofic.
\end{p}
\begin{proof}
Let $A$ be such an algebra. Consider $(L_k)_k$ an increasing
sequence of finite dimensional subspaces of $A$ such that $1\in
L_k$ and $A=\bigcup_kL_k$. Consider also $(\vp_k)_k$ a sequence of
strictly positive real numbers such that $\upsilon_k=\dim
L_k\cdot\vp_k\to_k 0$. For any $k$, by the definition of
amenability, there exists $S_k$ a finite dimensional subspace of
$A$ such that $\dim (aS_k\cap S_k)>(1-\vp_k)\cdot \dim S_k$ for any
$a\in L_k$. Then we can construct a linear map $\phi_k(a):S_k\to
S_k$ such that $\phi_k(a)$ is the left multiplication on a
subspace of dimension $(1-\vp_k)\cdot \dim S_k$. This implies that
$\phi_k:L_k\to End(S_k)$ is a $\upsilon_k$-almost representation
of $L_k$. As $\phi_k(a)$ is the left multiplication on a subspace
of dimension $(1-\vp_k)\cdot \dim S_k$ and $A$ has no zero divisors
it follows that $\rho(\phi_k(a))\geqslant 1-\vp_k$ for any $a\in
L_k$, $a\neq 0$.

Let $n_k=\dim S_k$ and construct $\Theta:A\to\Pi M_{n_k}/Ker
\rho_\omega$ by $\Theta=\Pi\phi_k/Ker\rho_\omega$. By construction
$\rho_\omega(\Theta(a))=1$ for $a\neq 0$. It follows that $\Theta$
is injective, so $A$ is linear sofic.
\end{proof}

The hypothesis of non-existence of zero divisors is too strong for
this proposition to hold. We can construct  a unital
morphism $\Theta$ for any amenable algebra. Therefore, amenable algebras have
almost finite dimensional representations. The non-existence of
zero divisors implies $\rho_\omega(\Theta(a))=1$, but we only use
$\rho_\omega(\Theta(a))>0$ for the injectivity of~$\Theta$.

It is easy to construct almost finite dimensional representations
for $LEF$ algebras (that is, algebras locally embeddable into
finite dimensional ones~\cite{VG:lef,Zi}) as also noticed in \cite{El}. In particular, any
amenable or $LEF$ algebra that is also simple is linear sofic.

There exist algebras that are not stably finite (see, for instance, Example \ref{ex:notsf} below). In particular,
such algebras are \emph{not linear sofic}. Combining Propositions
\ref{description of RR} and \ref{description of SR} we immediately
see that $RR(A)\subset SR(A)$. If $RR(A)\varsubsetneq SR(A)$ then
$A/RR(A)$ will be a stably finite non-linear sofic algebra. Such
algebras seem difficult to find as counterexamples to soficity in
general proved to be elusive.

\begin{ex}\label{ex:notsf}
Let us present an example of an algebra that is directly finite but
it is not stably finite. This construction is due to Sheperdson.
Let $A$ be the unital algebra over $F$ generated by elements
$\{x,y,z,t,a,b,c,d\}$ and relations
$\{xa+yc=1;xb+yd=0;za+tc=0;zb+td=1\}$. These relations are chosen
such that:
\[\left(%
\begin{array}{cc}
  x & y \\
  z & t \\
\end{array}%
\right)\left(%
\begin{array}{cc}
  a & b \\
  c & d \\
\end{array}%
\right)=Id_2\] Then $A$ is directly finite but it is not stably
finite. Details can be found in \cite[Exercise 1.18, p. 11]{La}. \end{ex}

\begin{ex}
In \cite{Cor} Cornulier constructed a sofic group that is not
initially sub-amenable. Its group algebra is linear sofic by
Theorem \ref{mainresult} and Proposition \ref{sofic implies
lsofic}. On the other hand, this algebra is neither LEF by Theorem 1 of
\cite{Zi}, nor amenable.
\end{ex}

\subsection{Computations of the sofic radical}

In this section, we prove that the rank radical is equal to the
sofic radical for group algebras. We also provide a
characterization of the sofic radical for a group algebra.

\begin{p}
Let $G$ be a countable group and let $LSR(G)$ be its linear sofic
radical. Denote by $f\colon G\to G/LSR(G)$ the canonical projection and
extend this morphism to group algebras: $\tilde f\colon \cz G\to
\cz(G/LSR_F(G))$. Then $SR(\cz G)=Ker\tilde f$.
\end{p}
\begin{proof}
Let $\Psi\colon\cz G\to\Pi M_{n_k}/Ker \rho_\omega$ be a unital
morphism. By Proposition \ref{invertible in ultraproduct}, $\Psi$
can be restricted to a morphism of the group, so $\Psi(u_g)=1$ if
$g\in LSR(G)$. Then $f(g)=f(h)$ implies $\Psi(u_g)=\Psi(u_h)$. Now
we can see that $Ker\tilde f\subset Ker\Psi$. As $\Psi$ was
arbitrary, we get $Ker\tilde f\subset SR(\cz G)$.

The group $G/LSR(G)$ is linear sofic, so by Theorem
\ref{mainresult} there exists $\Theta\colon\cz(G/LSR(G))\to\Pi
M_{n_k}/Ker \rho_\omega$ an injective unital morphism. Then
$\Theta\circ\tilde f\colon\cz G\to\Pi M_{n_k}/Ker \rho_\omega$ is a
unital morphism such that $Ker\Theta\circ\tilde f=Ker\tilde f$. It
follows that $SR(\cz G)\subset Ker\tilde f$.
\end{proof}

\begin{te}
For any group $G$ we have $SR(\cz G)=RR(\cz G)$.
\end{te}
\begin{proof}
Let $p\notin RR(\cz G)$ and assume that $p\in SR(\cz G)$. Let
$G=\bigcup_i B_i$ where $\{B_i\}_i$ is an increasing sequence of
finite subsets each containing the support of $p$ such that $1\in B_i$ and $B_i=B_i^{-1}$.

Let $\Phi_i\colon\cz(B_i\cup B_i^2)\to\Pi_{k\to\omega}
M_{n_{i,k}}/Ker\rho_\omega$ be a partial morphism such that
$\Phi_i(p)\neq 0$. Then $\Phi_i$ restricted to $B_i$ has its image
included in $\Pi_{k\to\omega}GL_{n_{i,k}}/d_\omega$. Now we can
apply Proposition \ref{large distance} to get a partial morphism
$\Psi_i:\cz(B_i)\to\Pi_{k\to\omega}M_{m_{i,k}}/Ker\rho_\omega$ such that for any
$g,h\in B_i$:
\begin{align*}
&\Phi_i(u_g)=\Phi_i(u_h)\Longrightarrow\Psi_i(u_g)=\Psi_i(u_h)\\
&\Phi_i(u_g)\neq\Phi_i(u_h)\Longrightarrow
d_\omega(\Psi_i(u_g),\Psi_i(u_h))\geqslant\frac14.
\end{align*}
We construct the ultraproduct of the family $\{\Psi_i\}_i$,
$\Psi=\Pi_{i\to\omega}\Psi_i/d_\omega$. Then $\Psi:\cz
G\to\Pi_{(i,k)\to\omega\otimes\omega} M_{m_{i,k}}/Ker
\rho_{\omega\otimes\omega}$ is a unital morphism. If
$\Psi(u_g)=\Psi(u_h)$ then $\lim_{i\to\omega}d_\omega
(\Psi_i(u_g),\Psi_i(u_h))=0$. The properties of $\Psi_i$ imply
that $\{i:\Psi_i(u_g)=\Psi_i(u_h)\}\in\omega$ in this case.

Let $f:G\to G/LSR(G)$ be the canonical projection used also in the
previous proposition. Then $f(g)=f(h)$ implies
$\Psi(u_g)=\Psi(u_h)$. As argued earlier $\Psi(u_g)=\Psi(u_h)$ iff
$\{i:\Psi_i(u_g)=\Psi_i(u_h)\}\in\omega$. Because the support of
$p$ is finite we can find an $i_0$ such that: $g,h\in supp\ p$ and
$f(g)=f(h)$ implies $\Psi_{i_0}(u_g)=\Psi_{i_0}(u_h)$. Then also
$\Phi_{i_0}(u_g)=\Phi_{i_0}(u_h)$.

By the previous proposition and our initial assumption that $p\in
SR(\cz G)$ we get $\tilde f(p)=0$. This implies that
$\Phi_{i_0}(p)=0$, which is a contradiction.
\end{proof}



\begin{thebibliography}{8}

\bibitem[B08]{Barth}
L. Bartholdi, \emph{On amenability of group algebras. I.}  
Israel J. Math. 168 (2008), 153--165. 

\bibitem[BCR98]{rag}
J. Bochnak, M. Coste, M.-F. Roy, \emph{Real algebraic geometry},  Ergebnisse der Mathematik und ihrer Grenzgebiete (3), 36, Springer-Verlag, Berlin, 1998.

\bibitem[CaP\u a12]{Ca-Pa}
V. Capraro and L. P\u aunescu, \emph{Product Between Ultrafilters and Applications
 to the Connes' Embedding Problem} J.Oper.Theory, 68(1) (2012), 165--172.

\bibitem[CS09]{Tullio}
T. Ceccherini-Silberstein and A. Samet-Vaillant, \emph{Asymptotic invariants of finitely generated algebras. A generalization of Gromov's quasi-isometric viewpoint,} Functional analysis. J. Math. Sci. (N. Y.) 156(1) (2009), 56--108. 

\bibitem[Cor11]{Cor}
Y. Cornulier, \emph{A sofic group away from amenable groups},
Math. Ann. 350(2) (2011), 269--275.

\bibitem[Del78]{Del}
P. Delsarte, \emph{Bilinear forms over a finite field, with applications to coding theory},
Journal of Combinatorial Theory A, 25 (1978), 226--241.

\bibitem[El03]{El:aa}
G. Elek, \emph{The amenability of affine algebras} 
J. Algebra 264(2) (2003), 469--478. 

\bibitem[El05]{El}
G. Elek, \emph{On algebras that have almost finite dimensional
representations}, J. Algebra, 4(2) (2005), 179--186.

\bibitem[ElSz04]{El-Sz1}
G. Elek and E. Szabo, \emph{Sofic groups and direct finiteness}, J. Algebra, 280(3) (2004),
426--434.

\bibitem[ElSz06]{El-Sz}
G. Elek and E. Szabo, \emph{On sofic groups}, J. Group Theory, 9(2) (2006),
161--171.

\bibitem[FaSh09]{Fa-Sh}
I. Farah and S. Shelah, \emph{A dichotomy for
the number of ultrafilters}, (2009), \textsc{arxiv:0912.0406}.

\bibitem[Gab85]{Gab}
E. Gabidulin, \emph{Theory of codes with maximum rank distance}, Problems on Information Transmission,  21(1) (1985), 1--12.

\bibitem[GlRi08]{Gl-Ri}
L. Glebsky and L. M. Rivera, \emph{Sofic groups and profinite topology on free
groups}, J. Algebra, 320(9) (2008), 3512--3518.

\bibitem[GlRi09]{Gl-Ri-ar}
L. Glebsky and L. M. Rivera, \emph{On low rank perturbations of complex matrices and some discrete metric spaces}, Electron. J. Linear Algebra 18 (2009), 302--316.

\bibitem[Gr99]{Gr}
M. Gromov, \emph{Endomorphism of symbolic
algebraic varieties}, J. Eur. Math. Soc., 1 (1999), 109--197.

\bibitem[Gr08]{Gr08}
M. Gromov, \emph{Entropy and isoperimetry for linear and non-linear group actions}, 
Groups Geom. Dyn., 2(4) (2008), 499--593. 

\bibitem[Hua45]{Hua:45}
L.-K. Hua,  \emph{Geometries of matrices. I. Generalizations of von Staudt's theorem},
Trans. Amer. Math. Soc., 57 (1945), 441--481.

\bibitem[IiIwa09]{Ii-Iw}
K. Iima and R. Iwamatsu, \emph{On the Jordan decomposition of
tensored matrices of Jordan canonical forms}, Math. J. Okayama
Univ., 51 (2009), 133--148.

\bibitem[La07]{La}
T. Lam, \emph{Exercise in Modules and Rings},
Problem books in mathematics, Springer, 2007.

\bibitem[Lu11]{Lu}
M. Lupini, \emph{Logic for metric structures and
the number of universal sofic and hyperlinear groups}, (2011),
\textsc{arxiv:1111.0729}, to appear in Proc. Amer. Math. Soc.

\bibitem[Ma40]{Ma}
A. Malcev, \emph{On isomorphic matrix
representations of infinite groups}, Mat. Sb., 8(50) (1940), 405--422.

\bibitem[MaVl]{Ma-Vl}
A. Martsinkovsky and A. Vlassov, \emph{The representation rings of
$k[x]$}, preprint, {http://www.math.neu.edu/~martsinkovsky/GreenExcerpt.pdf}.

\bibitem[Oz09]{Oz}
N. Ozawa, \emph{Hyperlinearity, sofic groups and applications to group theory}, (2009), {http://www.kurims.kyoto-u.ac.jp/~narutaka/notes/NoteSofic.pdf}.

\bibitem[Pe08]{Pe}
V. Pestov, \emph{Hyperlinear and sofic groups: a brief guide}, Bull.
Symbolic Logic, 14(4) (2008), 449--480.

\bibitem[PeKw09]{Pe-Kw}
V. Pestov and A. Kwiatkowska, \emph{An introduction to hyperlinear and sofic groups},
 (2009), \textsc{arxiv:0911.4266}.

\bibitem[R\u a08]{Ra}
F. R\u adulescu, \emph{The von Neumann algebras of the non-residually
finite Baumslag group $\langle a,b\mid ab^3a^{-1}=b^2\rangle$ embeds into
$R^{\omega}$}, Hot topics in operator theory, Theta Ser.
Adv. Math., 9, Theta, Bucharest, (2008), 173--185.

\bibitem[S00]{Samet}
A. Samet-Vaillant,  \emph{$C^*$-algebras, Gelfand-Kirillov dimension, and F\o lner sets} 
J. Funct. Anal. 171(2) (2000), 346--365. 

\bibitem[S06]{Semrl}
P. \v{S}emrl,  \emph{Maps on matrix and operator algebras}, Jahresber. Deutsch. Math.-Verein., 108(2) (2006), 91--103.

\bibitem[VG97]{VG:lef}
A. Vershik and E. Gordon,
\emph{Groups that are locally embeddable in the class of finite groups}, Algebra i Analiz 9(1) (1997), 71--97; translation in 
St. Petersburg Math. J. 9(1) (1998), 49--67. 

\bibitem[Wan96]{Wan}
Z.-X. Wan, \emph{Geometry of matrices}, World Scientific, Singapore, (1996).

\bibitem[W00]{W}
B. Weiss, \emph{Sofic groups and dynamical systems},
Ergodic theory and harmonic analysis (Mumbai, 1999), Sankhy\=a
Ser. A, 62(3) (2000), 350--359.

\bibitem[Zi02]{Zi}
M. Ziman, \emph{On finite approximations of
groups and algebras}, Illinois J. of Math., 43(3) (2002), 837--839.

\end{thebibliography}
\end{document}